\documentclass[preprint,10pt]{elsarticle}

\usepackage{mathrsfs,url,color}
\usepackage{graphicx}
\usepackage{amsthm,amsmath,amsfonts,epsfig,wasysym,enumerate,amssymb}
\usepackage{lineno,hyperref}
\usepackage[left=2cm, right=2cm, top=3cm]{geometry}

\journal{Stochastic Processes and their Applications}

\newtheorem{theorem}{Theorem}[section]
\newtheorem{proposition}[theorem]{Proposition}
\newtheorem{lemma}[theorem]{Lemma}
\newtheorem{corollary}[theorem]{Corollary}
\date{}

\theoremstyle{definition}

\newtheorem{problem}[theorem]{Open Problem}

\newtheorem{claim}[]{Claim}

\theoremstyle{remark}
\newtheorem{remark}[theorem]{Remark}

\numberwithin{equation}{section}

\newcommand{\E}{\mathbb{E}}

\newcommand{\R}{\mathbb{R}}
\def\var{\operatorname{Var}}		% variance
\def\cov{\operatorname{Cov}}		% variance
\def\dx{\operatorname{d}}

\begin{document}
	
	\begin{frontmatter}
		
		\title{On Rio's proof of limit theorems for dependent random fields}
		\author{L\^{e} V\v{a}n Th\`{a}nh}
		\address{Department of Mathematics, Vinh University, 182 Le Duan, Vinh, Nghe An, Vietnam }
		\ead{levt@vinhuni.edu.vn}

		\begin{abstract}
This paper presents an exposition of Rio's proof of the strong law of large numbers and extends his method to random fields. 
			In addition to considering the rate of convergence in the Marcinkiewicz--Zygmund strong law of large numbers, we go a step further by establishing (i) 
			the	Hsu--Robbins--Erd\"{o}s--Spitzer--Baum--Katz theorem, (ii) the Feller weak law of large numbers, and (iii) the Pyke--Root theorem on mean convergence for dependent random fields. 
			These results significantly improve several particular cases in the literature. The proof is based on new maximal inequalities that hold for random fields 
			satisfying a very general dependence structure.
		\end{abstract}
		
		
	\end{frontmatter}

	\section{Introduction and Main Results}\label{sec.intro}

	Consider a sequence $\{X_{n},n\ge1\}$ of square integrable, mean zero random variables. Let $S_{n}=X_1+\cdots+X_n$, $n\ge 1$ be
	the partial sums. Many dependence structures possess the following inequality:
	\begin{equation}\label{orthogonal.inequality}
		\E S_{n}^2 \le C\sum_{i=1}^{n}\E X_{i}^2,\ n\ge1.
	\end{equation}
	Here and hereafter, the symbol $C$ denotes an absolute constant which is not necessarily the same one in each appearance.
To prove strong laws of large numbers (SLLN), we usually need
	a stronger inequality which will be referred to as a Kolmogorov--Doob-type maximal inequality:
	\begin{equation}\label{kd.inequality}
		\E \left(\max_{1\le k\le n} S_{k}^2\right) \le C\sum_{i=1}^{n}\E X_{i}^2,\ n\ge1.
	\end{equation}
	However, \eqref{kd.inequality} is not available for some interesting dependence structures, such as negative dependence,
	extended negative dependence or various mixing sequences. It is even
	invalid for pairwise independence or pairwise negative dependence. Therefore, stronger conditions are usually required for the SLLN under these dependence structures compared to the independence case (see, e.g., Cs\"{o}go et al. \cite{csorgo1983strong} and Mart\u{\i}kainen \cite{martikainen1995strong}). In 1981, Etemadi \cite{etemadi1981elementary} proved that the Kolmogorov SLLN still holds for the pairwise independent and identical distribution (p.i.i.d.) case. The Etemadi subsequences method, however, does not seem to work when the norming sequences are of the form $b_{n}=o(n)$, as in the case of the Marcinkiewicz--Zygmund SLLN (see Remark 3 of Janisch \cite{janisch2021kolmogorov}). Cs\"{o}go et al. \cite{csorgo1983strong} showed that under pairwise independence, the Kolmogorov SLLN for the non-identical distribution case does not hold in general. 
	
	In some cases, it may be necessary to bound moments of order higher than $2$ for either the partial sums or the maximum of the partial sums.
	Let $p\ge 2$ and let $\{X_i,1\le i\le n\}$ be a collection of independent mean zero random variables. The Rosenthal inequality states that
	\begin{equation}\label{rosenthal.inequality1}
		\E\left|\sum_{i=1}^{n}X_i\right|^p\le C(p)\max\left\{\sum_{i=1}^{n}\E |X_{i}|^{p},\left(\sum_{i=1}^{n}\E X_{i}^{2}\right)^{p/2}\right\}.
	\end{equation}
Hereafter, $C(p)$ is a constant depending only on $p$.
	Johnson et al. \cite{johnson1985best} proved that if the random variables $X_{i}$, $1\le i\le n$ are  independent and symmetric,
	then \eqref{rosenthal.inequality1} holds with $C(p)=\left(\frac{K p}{\log p}\right)^{p}$,
	where $K$ is a constant satisfying $1/(e\sqrt{2})\le K\le 7.35$. 
	Recently, Chen et al. \cite{chen2021error} used Stein’s method
	and obtained the bound $K\le 3.5$ without assuming the symmetry of the random
	variables. It is noteworthy that
	the rate $p/\log p$ in the expression of $C(p)$ is optimal, as shown by Johnson at al. \cite{johnson1985best}.
	A stronger version of \eqref{rosenthal.inequality1}
	is
	\begin{equation}\label{rosenthal.inequality2}
		\E\left(\max_{k\le n}\left|\sum_{i=1}^{k}X_i\right|^{p}\right)\le C(p)\max\left\{\sum_{i=1}^{n}\E |X_{i}|^{p},\left(\sum_{i=1}^{n}\E X_{i}^{2}\right)^{p/2}\right\}
	\end{equation}
	which plays a crucial tool in the proof of many limit theorems (see, e.g., \cite{cuny2014martingale,merlevede2013rosenthal,peligrad2005new,utev2003maximal}).
	We will refer to \eqref{rosenthal.inequality2} as a Rosenthal-type maximal inequality.
 Rosenthal-type maximal inequalities have been established for various dependence structures,
 such as stationary sequences (Merlevede and Peligrad \cite{merlevede2013rosenthal}, Peligrad and Utev \cite{peligrad2005new}),
 $\rho$-mixing sequences (Shao \cite{shao1995maximal}), negatively associated sequences (Shao \cite{shao2000comparison}),
and $\rho^{*}$-mixing sequences (Peligrad and Gut \cite{peligrad1999almost}, Utev and Peligrad \cite{utev2003maximal}), etc.
	
	Let $\{X_{n},n\ge1\}$ be a sequence of independent and identically distributed (i.i.d.) random variables. Hsu and Robbins \cite{hsu1947complete} proved that if $\E X_1=0$ and 
	$\E X_{1}^2<\infty$, then the sample mean converges to $0$ completely, i.e.,
	\begin{equation}\label{hre.complete}
		\sum_{n=1}^{\infty}\mathbb{P}\left(\left|\sum_{i=1}^{n}X_i\right|>\varepsilon n \right)<\infty \ \text{ for all }
		\ \varepsilon >0.
	\end{equation}
	Erd\"{o}s \cite{erdos1949theorem} proved that the converse also holds, i.e., \eqref{hre.complete}
	implies $\E X_1=0$ and 
	$\E X_{1}^2<\infty$. This famous result was extended to the case where 
	$\E X_{1}^2$ can be infinite by Baum and Katz \cite{baum1965convergence}. 
	The Baum--Katz theorem reads as follows.
	
	\begin{theorem}[Baum and Katz \cite{baum1965convergence}]\label{bk}
		Let $p\ge 1, 1/2<\alpha\le 1, \alpha p\ge 1$ and let $\{X_{n},n\ge1\}$ be a sequence of i.i.d. random variables. If
		\begin{equation}\label{bk01}
			\E X_1=0 \ \text{ and }\ \E |X_{1}|^p<\infty,
		\end{equation}
		then
		\begin{equation}\label{bk03}
			\sum_{n=1}^{\infty}n^{\alpha p-2}\mathbb{P}\left(\left|\sum_{i=1}^{n}X_i\right|>\varepsilon n^{\alpha}\right)<\infty \ \text{ for all }
			\ \varepsilon >0,
		\end{equation}
		and
		\begin{equation}\label{bk05}
			\sum_{n=1}^{\infty}n^{\alpha p-2}\mathbb{P}\left(\max_{k\le n}\left|\sum_{i=1}^{k}X_i\right|>\varepsilon n^{\alpha}\right)<\infty \ \text{ for all }
			\ \varepsilon >0.
		\end{equation}
		Conversely, if one of the sums is finite for all $\varepsilon>0$, then \eqref{bk01} holds.
	\end{theorem}
	The implication $\eqref{bk05} \Rightarrow \eqref{bk03}$ is trivial and the
	implication $\eqref{bk03} \Rightarrow \eqref{bk05}$ is a direct consequence of the L\'{e}vy inequalities (see, e.g., \cite[Theorem 3.7.1]{gut2013probability}) as noted by Gut and Stadtm\"{u}ller \cite[Page 447]{gut2012hsu}.
	The equivalence of \eqref{bk01} and \eqref{bk03} for the case where $p=1$ and $\alpha=1$ was proved by Spitzer \cite{spitzer1956combinatorial}.
	The case where $p>1$, $1/2<\alpha\le 1$ and $\alpha p>1$ is
	the first part of Theorem 3 of Baum and Katz \cite{baum1965convergence},
	and it reduces to the Hsu--Robbins--Erd\"{o}s theorem when $p=2$ and $\alpha=1$. 
	The case where $1\le p<2$ and $\alpha=1/p$ is the second part of Theorem 1 of Baum and Katz \cite{baum1965convergence}, and it is of special interest because 
	each of \eqref{bk01}, \eqref{bk03} and \eqref{bk05} is equivalent to the
	Marcinkiewicz--Zygmund SLLN. 
For the case $0<p<1$,
	Peligrad \cite{peligrad1985convergence} proved that the second half of
	\eqref{bk01} implies \eqref{bk05} without assuming any dependence structure (see Peligrad \cite[Theorem 1]{peligrad1985convergence}).

	In \cite{pyke1968convergence}, Pyke and Root
	proved that if $1\le p<2$, then the condition \eqref{bk01} is also necessary and sufficient for
	convergence in $\mathcal{L}_p$ of the partial sums. 
	
	\begin{theorem}[Pyke and Root \cite{pyke1968convergence}]\label{pr}
		Let $1\le p<2$ and let $\{X_{n},n\ge1\}$ be a sequence of i.i.d. random variables. Then
		\begin{equation}\label{bkpr03}
			\frac{\sum_{i=1}^{n}X_{i}}{n^{1/p}}\overset{\mathcal{L}_p}{\to} 0 \text{ as } n\to\infty
		\end{equation}
		if and only if \eqref{bk01} holds.
	\end{theorem} 
	
	The Hsu--Robbins--Erd\"{o}s--Spitzer--Baum--Katz theorem was extended in various directions. 
	We refer to 
	\cite{dedecker2008convergence,gut1978marcinkiewicz,gut2012hsu,kuczmaszewska2011convergence,peligrad1999almost,peligrad2023convergence,rio1995maximal,shao1995maximal,stoica2011note}
	and the references therein. In all these papers, the 
	maximal inequalities play a crucial step in the proofs.
	It was shown that
	if a sequence of random variables satisfies a Kolmogorov--Doob-type maximal inequality, then the Baum--Katz theorem holds for the case where $1\le p<2$ (see, e.g., \cite{thanh2023new}).
	On the Pyke--Root theorem, however, no maximal inequality is needed and the result holds for sequences of p.i.i.d. random variables (see, e.g., Chen, Bai, and Sung \cite{chen2014bahr}).
	
	In \cite{rio1995vitesses}, 	Rio developed a new method to prove that the Baum--Katz theorem (for the case $1\le p<2$ and $\alpha=1/p$) still holds for
	sequences of p.i.i.d. random variables.
Although the maximal inequalities are somewhat concealed in Rio's proof \cite{rio1995vitesses},
	his method provides an elegant way to bound 
	the tail probabilities of the maximum of partial sums of pairwise independent random variables.
Rio's method has recently been applied by Th\`{a}nh \cite{thanh2020theBaum, thanh2023extension, thanh2023hsu} 
to derive laws of large numbers with regularly varying norming constants.	In this paper, we give an exposition of Rio's proof by 
	showing that his method can lead to a Rosenthal-type maximal inequality
	for double sums of dependent random variables. This result is then used
	to prove various limit theorems for two-dimensional random fields.
	In addition to extending Rio's result on SLLN for dependent random fields, 
	we also obtain the Feller weak law of large numbers (WLLN) and the Pyke--Root theorem on mean convergence for the maximum of double sums 
	of dependent random variables. Furthermore, the Hsu--Robbins--Erd\"{o}s SLLN for the maximum of double sums from double arrays of dependent random variables
	is established. It is important to note that the Hsu--Robbins--Erd\"{o}s theorem does not hold in general if the
	independence assumption is weakened to
	the pairwise independence, even when the underlying random variables are uniformly bounded
	(see Szynal \cite{szynal1993complete}).
	We note further that in the proof of the Pyke--Root theorem and the Feller WLLN for partial sums, as mentioned before, no maximal inequalities
	are required and
	the results hold for p.i.i.d. random variables.
	However, if one considers convergence of the maximum of partial sums, a Kolmogorov--Doob-type maximal inequality would be needed, and the existing methods do not seem
	to push through for the case of p.i.i.d. random variables.

	Wichura \cite{wichura1969inequalities} was apparently the first to 
	establish the following multidimensional version of the Kolmogorov--Doob-type maximal inequality \eqref{kd.inequality}
	for the case of independent random variables. Let $\{X_{m,n},m\ge 1,n\ge 1\}$ be a double array of independent mean zero random variables and let
	$S_{m,n}=\sum_{i=1}^{m}\sum_{j=1}^{n}X_{i,j}$
	be the partial sums. Then
	\begin{equation}\label{kd.inequality.doule}
		\E \left(\max_{k\le m, \ell\le n}S_{k,\ell}^2\right) \le 16\sum_{i=1}^{m}\sum_{j=1}^{n}\E X_{i,j}^2,\ m\ge 1, n\ge1.
	\end{equation}
	For moment inequalities of the partial sums \eqref{orthogonal.inequality} and \eqref{rosenthal.inequality1}, it is clear that the case of the single sums is 
	the same as its double sums counterpart. However,
	there is a substantial difference between \eqref{kd.inequality.doule} and \eqref{kd.inequality}
	because of the partial (in lieu of linear) ordering of the index set
	$\{(i,j),i\ge1,j\ge1\}$. Wichura's \cite{wichura1969inequalities} results
	had a great impact on the investigation of limit theorems for random fields.
	For the case of i.i.d. random variables, we refer to a survey paper by Pyke \cite{pyke1973partial} which covers many important topics 
	such as fluctuation
	theory, the SLLNs, inequalities, the central limit theorems, and the law of the
	iterated logarithm for the multidimensional sums.
	For a comprehensive exposition on the 
	limit theorems for multiple sums of independent
	random variables, we refer to a monograph by Klesov \cite{klesov2014limit}.
	
	The Hsu--Robbins--Erd\"{o}s--Spitzer--Baum--Katz and the Pyke--Root theorems were extended to independent random fields by Gut \cite{gut1978marcinkiewicz,gut1980convergence}
	and Gut and Stadtm\"{u}ller \cite{gut2012hsu}, and to dependent random fields by Peligrad and Gut \cite{peligrad1999almost}, 
	Giraudo \cite{giraudo2019deviation} and 
	Kuczmaszewska and Lagodowski \cite{kuczmaszewska2011convergence}, among others.
	The dependence structures considered in Peligrad and Gut \cite{peligrad1999almost}, Giraudo \cite{giraudo2019deviation} and 
	Kuczmaszewska and Lagodowski \cite{kuczmaszewska2011convergence} are, respectively, $\rho^{*}$-mixing random fields, martingale differences random fields, and negatively associated random fields,
all possessing a Kolmogorov--Doob-type maximal inequality. When working with limit theorems for the maximum of multidimensional sums of dependent random variables, we
	encounter the following difficulties:
	\begin{description}
		\item[(i)] The Kolmogorov--Doob-type and the Rosenthal-type maximal inequalities are not valid, even in the case of dimension one (e.g., pairwise independence, pairwise negative dependence).
		This is due to the fact that the Kolmogorov SLLN for the non-identically distributed
		case does not necessarily hold if the underlying random variables are only pairwise independent (see, e.g., Cs\"{o}go et al. \cite[Theorem 3]{csorgo1983strong}).
		
		\item[(ii)] For some dependence structures, the Kolmogorov--Doob-type and the Rosenthal-type maximal inequalities are not available for the multidimensional setting 
		(e.g., the $\rho'$-mixing random fields, negatively dependent random fields).
	\end{description}

	The advantage of our approach is that we only assume that
	the underlying random variables satisfy \eqref{rosenthal.inequality1} for some fixed $p\ge 2$.
	Therefore, we can avoid the above difficulties, and the main results can be applied to all aforementioned dependence structures.
	
	For the sake of clarity, especially due to the complicated notation, we shall establish
	the results for double-indexed random fields. The results would be able to extend to $d$-dimensional random fields for any integer $d\ge2$ by the same method.
	
	Throughout this paper, $C(\cdot)$, $C_1(\cdot),\ldots$ denote generic constants which are not necessarily the same one in each appearance,
	and depend only on the variables inside the parentheses.
	For $a,b\in\R$, $\max\{a,b\}$ will be denoted by $a\vee b$, and 
	the natural logarithm of $a\vee 2$ will be denoted by $\log a$.
	For a set $S$, ${\textbf{1}}(S)$ denotes the indicator
	function of $S$, and $|S|$ denotes the cardinality of $S$. For $x\ge0$, and for a fixed positive integer $\nu$, we let
	\begin{equation}\label{notation1}
		\log_\nu(x):=(\log x)(\log\log x)\ldots (\log\cdots\log x),
	\end{equation}
	and
	\begin{equation}\label{notation2}
		\log_{\nu}^{(2)}(x):=(\log x)(\log\log x)\ldots (\log\cdots\log x)^{2},
	\end{equation}
	where in both \eqref{notation1} and \eqref{notation2}, there are $\nu$ factors.
	For example, $\log_2(x)=(\log x)(\log\log x)$, and $\log_{3}^{(2)}(x)=(\log x)(\log\log x)(\log\log\log x)^{2}$, and so on.
	For positive sequences $\{u_n,n\ge1\}$ and $\{v_n,n\ge1\}$, we write $u_n\asymp v_n$
	to mean
	\[0<\liminf \dfrac{u_n}{v_n} \le \limsup \dfrac{u_n}{v_n} <\infty. \]
	
	The Hsu--Robbins--Erd\"{o}s--Spitzer--Baum--Katz, the Feller WLLN and the Pyke--Root theorems
	were originally stated for identically distributed random variables. 
	A natural extension of the identical distribution condition, known as stochastic domination, is defined as follows.
	A family of random variables $\{X_{\lambda},\lambda\in \Lambda\}$ is said to be \textit{stochastically
		dominated} by a random variable $X$ if
	\begin{equation*}
		\sup_{\lambda\in\Lambda}\mathbb{P}(|X_\lambda|>x)\le \mathbb{P}(|X|>x), \ x\in\mathbb{R}.
	\end{equation*}
	Some interesting properties concerning the concept of stochastic domination as well as relationships between stochastic domination and 
	uniform integrability were recently established in \cite{rosalsky2021note}. If 
	$\{X_{\lambda},\lambda\in \Lambda\}$ is stochastically
	dominated by a random variable $X$, then for all $r>0$ and $a>0$,
	\begin{equation}\label{sto.domination1}
		\sup_{\lambda\in\Lambda}\E\left(|X_{\lambda}|^r\mathbf{1}(|X_{\lambda}|>a)\right)\le \E\left(|X|^r\mathbf{1}(|X|>a)\right)
	\end{equation}
	and
	\begin{equation}\label{sto.domination2}
		\sup_{\lambda \in \Lambda}\E(|X_{\lambda}|^r \mathbf{1}(|X_{\lambda}|\le a))\le \E(|X|^r\mathbf{1}(|X|\le a))+ a^r\mathbb{P}(|X|>a)\le \E |X|^r.
	\end{equation}
	We will use \eqref{sto.domination1} and \eqref{sto.domination2} in our proofs without further mention.
	
	In this paper, we consider a very general dependence structure, defined as follows:
	
\textbf{Condition $(H_{2q})$.}
	Let $q\ge 1$ be a real number. A family of random variables $\{X_{\lambda},\lambda\in \Lambda\}$ is said to satisfy \textit{Condition $(H_{2q})$} if
	for all finite subset $I\subset \Lambda$ and for all family of increasing functions $\{f_{\lambda},\lambda\in I\}$, there exists a finite constant $C(q)$ depending only on $q$ such that
	\begin{equation}\label{eq.condH2q}
		\E\left|\sum_{\lambda\in I}\left(f_{\lambda}(X_{\lambda})-\E f_{\lambda}(X_{\lambda})\right)\right|^{2q}\le C(q) \left(|I|\max_{\lambda\in I}\E |f_{\lambda}(X_{\lambda})|^{2q}+|I|^{q}\max_{\lambda\in I}\left(\E f_{\lambda}^2(X_{\lambda})\right)^{q}\right)
	\end{equation}
	provided the expectations are finite. 
	
	It is easy to see that if $\{X_{\lambda},\lambda\in\Lambda\}$ is a family of pairwise independent (resp, quadruple-wise independent) random variables,
	then it satisfies Condition $(H_2)$ (resp., Condition $(H_4)$). We would like to note that for most of the results on laws of large numbers, 
	we only need to assume that the underlying random variables satisfy Condition $(H_{2})$. 
	By Theorem 2.1 of Chen and Sung \cite{chen2020rosenthal}, we see that
	if a collection of random variables satisfies Condition $(H_{2q'})$ for some $q'> q\ge 1$, then it satisfies Condition $(H_{2q})$.
Various dependence structures satisfy Condition $(H_{2q})$ for all $q\ge 1$
	such as negative dependence, extended negative dependence (see Lemmas 2.1 and 2.3 of Shen et al. \cite{shen2017weak}), $\rho^*$-mixing (see Theorem 4 of Peligrad and Gut \cite{peligrad1999almost}), and $\rho'$-mixing (see Theorem 29.30 of Bradley \cite{bradley2007introduction}). 
	A more detailed discussion of these dependence structures will be provided in Subsection \ref{subsec.mixing}.
	It is worth noting that pairwise negative dependence satisfy 
	Condition $(H_{2})$, but it does not meet Condition $(H_{2q})$ for $q \geq 2$ (see Example on pages 145--146 in Szynal \cite{szynal1993complete} and the discussion on page 2 in Th\`{a}nh \cite{thanh2023hsu}).
To the best of our knowledge, \eqref{eq.condH2q} is not available for $\alpha$-mixing random variables even for $q=1$.
We refer to Chapter 1 of Rio \cite{rio2017asymptotic} for several bounds of variance of the partial sums of $\alpha$-mixing random variables.
	
	The following theorem is the first main result of this paper.
	Theorem \ref{thm.BKHRE} is the Hsu--Robbins--Erd\"{o}s--Spitzer--Baum--Katz theorem for the maximum of
double sums of random variables satisfying Condition $(H_{2q})$.

	\begin{theorem}\label{thm.BKHRE}
		Let $p\ge 1$, $1/2<\alpha\le 1$, $\alpha p\ge 1$ and let 
		$\{X_{m,n}, \ m\ge 1,\ n \geq 1\}$ be a double array of random variables.
		Assume that the array $\{X_{m,n}, \ m\ge 1,\ n \geq 1\}$ satisfies Condition $(H_{2q})$ with $q=1$ if $1\le p<2$ and $q>(\alpha p-1)/(2\alpha-1)$ if $p\ge 2$.
		If $\{X_{m,n}, \ m\ge 1,\ n \ge 1\}$ is stochastically dominated by a
		random variable $X$ satisfying
		\begin{equation}\label{mz04}
			\E\left(|X|^p \log |X|\right)<\infty,
		\end{equation}
		then 
		\begin{equation}\label{mz05}
			\sum_{m=1}^{\infty}\sum_{n=1}^{\infty}
			(mn)^{\alpha p-2}\mathbb{P}\left(\max_{\substack{1\le u\le m\\ 1\le v\le n}}\left|
			\sum_{i=1}^{u}\sum_{j=1}^{v}(X_{i,j}-\E X_{i,j})\right|>\varepsilon (mn)^{\alpha}\right)<\infty \text{ for all }\varepsilon>0.
		\end{equation}
		Conversely, if $X_{m,n}, \ m\ge 1,\ n \ge 1$ have the same distribution as a random variable $X$ and for some $\mu\in\R$,
		\begin{equation}\label{mz05.converse}
			\sum_{m=1}^{\infty}\sum_{n=1}^{\infty}
			(mn)^{\alpha p-2}\mathbb{P}\left(\max_{\substack{1\le u\le m\\ 1\le v\le n}}\left|
			\sum_{i=1}^{u}\sum_{j=1}^{v}(X_{i,j}-\mu)\right|>\varepsilon (mn)^{\alpha}\right)<\infty \text{ for all }\varepsilon>0,
		\end{equation}
		then $\E X=\mu$ and \eqref{mz04} holds.
	\end{theorem}
	
	The proof of Theorem \ref{thm.BKHRE} will be presented in Section \ref{sec.bk}.	Similar to the case of dimension one (see, e.g., Remark 1 in \cite{dedecker2008convergence}), we have the following remark.
	
	\begin{remark}\label{rem:slln1}
		For arbitrary array $\{X_{m,n},m\ge1,n\ge1\}$ of integrable random variables, by writing
		\[	
		\begin{split}
			&\sum_{m=1}^{\infty}\sum_{n=1}^{\infty}
			(mn)^{\alpha p-2}\mathbb{P}\left(\max_{\substack{1\le u\le m\\ 1\le v\le n}}\left|
			\sum_{i=1}^{u}\sum_{j=1}^{v}(X_{i,j}-\E X_{i,j})\right|>\varepsilon (mn)^{\alpha}\right)\\
			&=\sum_{k=1}^{\infty}\sum_{\ell=1}^{\infty}\sum_{m=2^{k-1}}^{2^{k}-1}\sum_{n=2^{\ell-1}}^{2^{\ell}-1}
			(mn)^{\alpha p-2}\mathbb{P}\left(\max_{\substack{1\le u\le m\\ 1\le v\le n}}\left|
			\sum_{i=1}^{u}\sum_{j=1}^{v}(X_{i,j}-\E X_{i,j})\right|>\varepsilon (mn)^{\alpha}\right),
		\end{split}
		\]
		we easily prove that
		\eqref{mz05} is equivalent to
		\begin{equation}\label{mz02}
			\sum_{k=1}^{\infty}\sum_{\ell=1}^{\infty}
			2^{(k+\ell)(\alpha p-1)}\mathbb{P}\left(\max_{u< 2^k,v< 2^\ell}\left|\sum_{i=1}^{u}\sum_{j=1}^{v}(X_{i,j}-\E X_{i,j})\right|>\varepsilon 2^{(k+\ell)\alpha}\right)<\infty \text{ for all } \varepsilon>0.
		\end{equation}
		Since $\alpha p\ge 1$, \eqref{mz02} together with the Borel--Cantelli lemma imply
		\begin{equation*}
			\lim_{k\vee \ell\to\infty}\dfrac{\max_{1\le u< 2^k,1\le v< 2^\ell}\left|\sum_{i=1}^{u}\sum_{j=1}^{v}(X_{i,j}-\E X_{i,j})\right|}{2^{(k+\ell)\alpha}}=0 \ \text{ almost surely (a.s.),}
		\end{equation*}
		which, in turn, implies
		\begin{equation}\label{erem3}
			\lim_{m\vee n\to\infty}\dfrac{\max_{1\le u\le m,1\le v\le n}\left|\sum_{i=1}^{u}\sum_{j=1}^{v}(X_{i,j}-\E X_{i,j})\right|}{(mn)^{\alpha}}=0 \ \text{ a.s.}
		\end{equation}
		If $1\le p<2$, then by choosing $\alpha=1/p$ in \eqref{erem3}, we obtain the Marcinkiewicz--Zygmund SLLN.
	\end{remark}

	We will now present the Feller WLLN and the Pyke--Root theorem for the maximum of double sums of random variables satisfying Condition $(H_{2})$.
	For the Feller WLLN for partial sums from sequences of i.i.d. random variables, we refer to Feller \cite[Theorem 1, Section VII.7]{feller1971introduction}. The proofs
	of Theorems \ref{thm.Feller} and \ref{thm.PR} are presented in Section \ref{sec.feller-pyke-root}.
	
	\begin{theorem}\label{thm.Feller}
		Let $1\le p<2$ and let 
		$\{X_{m,n}, \ m\ge 1,\ n \geq 1\}$ be a double array of random variables satisfying Condition $(H_{2})$.
		For $n\ge1,i\ge1,j\ge1$, set
		\[b_{n}=n^{1/p},\ Z_{n,i,j}=X_{i,j}\mathbf{1}\left(|X_{i,j}|\le b_{n}\right).\]
		If $\{X_{m,n}, \ m\ge 1,\ n \geq 1\}$ is stochastically dominated by a
		random variable $X$ satisfying
		\begin{equation}\label{fel03}
			n\mathbb{P}(|X|>n^{1/p})\to 0 \text{ as }n\to\infty,
		\end{equation}	
		then 
		\begin{equation}\label{fel04}
			\frac{\max_{u\le m,v\le n}\left|\sum_{i=1}^{u}\sum_{j=1}^{v}(X_{i,j}-\E Z_{mn,i,j})\right|}{(mn)^{1/p}}\overset{\mathbb{P}}{\to} 0 \text{ as }m\vee n\to\infty.
		\end{equation}
		Conversely, if $X_{m,n},m\ge1,n\ge1$ are symmetric and have the same distribution as a random variable $X$, then \eqref{fel04} implies \eqref{fel03}.
	\end{theorem}

	\begin{theorem}\label{thm.PR}
		Let $1\le p<2$ and let 
		$\{X_{m,n}, \ m\ge 1,\ n \geq 1\}$ be a double array of random variables satisfying Condition $(H_{2})$.
		If $\{X_{m,n}, \ m\ge 1,\ n \geq 1\}$ is stochastically dominated by a random variable $X$ satisfying
		\begin{equation}\label{pr03}
			\E |X|^{p}<\infty,
		\end{equation}
		then 
		\begin{equation}\label{pr04}
			\frac{\max_{u\le m,v\le n}\left|\sum_{i=1}^{u}\sum_{j=1}^{v}(X_{i,j}-\E X_{i,j})\right|}{(mn)^{1/p}}\overset{\mathcal{L}_p}{\to} 0 \text{ as }m\vee n\to\infty.
		\end{equation}
		Conversely, if the random variables $X_{m,n},m\ge1,n\ge1$ have the same distribution functions as a random variable $X$ and
		\begin{equation}\label{pr05}
			\frac{\max_{u\le m,v\le n}\left|\sum_{i=1}^{u}\sum_{j=1}^{v}(X_{i,j}-\mu)\right|}{(mn)^{1/p}}\overset{\mathcal{L}_p}{\to} 0 \text{ as }m\vee n\to\infty
		\end{equation}
		for some real number $\mu$, then $\E X=\mu$ and \eqref{pr03} holds.
	\end{theorem}
	
	\begin{remark}
		Since quadruple-wise independent random variables satisfy Condition $(H_{4})$,
		by applying Theorem \ref{thm.BKHRE} for the case where
		$p=2$, $\alpha=1$ and $q=2$, we obtain the Hsu--Robbins--Erd\"{o}s theorem for the maximum of
		partial sums from a double array of quadruple-wise independent and identically distributed (q.i.i.d.) random variables, that is, if 
		$\{X_{m,n}, \ m\ge 1,\ n \ge 1\}$ a double array of q.i.i.d. random variables, then 
		\begin{equation*}\label{hr03.main}
			\E X_{1,1}=0\ \text{ and }\ \E\left(X_{1,1}^2\log |X_{1,1}|\right)<\infty
		\end{equation*}
		if and only if
		\begin{equation*}\label{hr05.main}
			\sum_{m=1}^{\infty}\sum_{n=1}^{\infty}\mathbb{P}\left(\max_{u\le m,v\le n}\left|\sum_{i=1}^{u}\sum_{j=1}^{v}X_{i,j}\right|>\varepsilon mn\right)<\infty \text{ for all } \varepsilon>0.
		\end{equation*}
		Similarly, we can apply Theorems \ref{thm.BKHRE}, \ref{thm.Feller} and \ref{thm.PR}
		to obtain the necessary and sufficient conditions for (i) the Spitzer--Baum--Katz theorem (for the case where $1\le p<2$), (ii) the Feller WLLN, and (iii) the Pyke--Root theorem for 
		double arrays of p.i.i.d. random variables.
	\end{remark}

	\begin{remark}\label{rem.open}
		
		The Reviewer kindly raised a question about the possibility of obtaining additional results for
$2q$-tuplewise independent random fields (see, e.g., \cite{bradley2012possible} for the definition) when $q\geq 3$. To establish the validity of 
		Theorem \ref{thm.BKHRE} for double arrays of $2q$-tuplewise independent random variables, 
		we would need to show that a collection of $2q$-tuplewise independent random variables satisfies Condition $(H_{2q})$. 
		Unfortunately, we are unable to achieve this even in the case of $q=3$. 
		We present it here as an open problem for future research.
	\end{remark}
	
	The rest of the paper is organized as follows. 
	In Section \ref{sec.maximal.inequalities}, 
	we use Rio's technique to establish some maximal inequalities
	for double sums of dependent random variables. 
	The proof of Theorem \ref{thm.BKHRE} is presented in Section \ref{sec.bk}.
	Section \ref{sec.feller-pyke-root} contains the proof of Theorems \ref{thm.Feller} and \ref{thm.PR}.
	Section \ref{sec:remarks} presents some corollaries and remarks.
	Finally, some technical results are proved in the Appendix.

\section{New maximal inequalities for double sums of dependent random variables}\label{sec.maximal.inequalities}

As mentioned in Section \ref{sec.intro}, although the maximal inequalities are ``almost hidden'' in the proof of Rio \cite{rio1995vitesses},
his method can lead to a new maximal inequality for pairwise independent random variables.
A brief discussion about Rio's technique in dimension one is given as follows. For simplicity, we assume that
$\{X_n,n\ge1\}$ is a sequence of p.i.i.d. integrable random variables.
Let $1\le p<2$, $b_n=n^{1/p}$ and $X_{n,i}=X_{i}\mathbf{1}(|X_i|\le b_n)$, $1\le i\le n$, $n\ge1$.  
When proving limit theorems such as the Baum--Katz theorem or
the Pyke--Root theorem, it suffices to control the tail probability of the form
\begin{equation}\label{21}
	\mathbb{P}\left(\max_{1\le j <2^n}\left|S_{n,j}\right|>\varepsilon b_{2^n}\right),
\end{equation}
where $\varepsilon>0$ and $S_{n,j}=\sum_{i=1}^j (X_{2^n,i}-\E X_{2^n,i}),\ n\ge 0.$
Since the random variables are only assumed to be pairwise independent, we would not be able to apply the Kolmogorov maximal inequality.
In \cite{rio1995vitesses}, Rio used the telescoping sums:
\begin{equation}\label{23}
	S_{n,j}=\sum_{m=1}^{n}(S_{m-1,j_{m-1}}-S_{{m-1},j_{m}})+\sum_{m=1}^{n}(S_{m,j}-S_{m-1,j}-S_{m,j_m}+S_{m-1,j_m}),  1\le j<2^n, 0\le m\le n, n\ge 1,
\end{equation}
where $S_{m,0}=0$ and $j_m = \lfloor j/2^m\rfloor \times 2 ^m$.
For the first term on the right hand side of \eqref{23}, we have
\begin{equation}\label{24}
	\left|S_{m-1,j_{m-1}}-S_{{m-1},j_{m}}\right|\le \left|\sum_{i=j_m+1}^{j_m+2^{m-1}}\left(X_ {2^{m-1},i}-\E X_{2^{m-1},i}\right)\right|.
\end{equation}
From \eqref{23}, \eqref{24} and the definition of $j_m$, we can address the problem of bounding the tail probability in \eqref{21} by bounding
\begin{equation}\label{25}
	I=\mathbb{P}\left(\sum_{m=1}^n \max_{0\le k<2^{n-m}}\left|\sum_{i=k2^m+1}^{k2^m+2^{m-1}}\left(X_ {2^{m-1},i}-\E X_ {2^{m-1},i}\right)\right| \ge \varepsilon b_{2^{n}}\right).
\end{equation}
By writing $b_{2^n}=\sum_{m=1}^n \lambda_{n,m}$ with suitable choices of $\lambda_{n,m}$,
and using Chebyshev's inequality and the p.i.i.d. assumption, we have
\begin{equation*}\label{26}
	\begin{split}
		I&\le \varepsilon^{-2}\sum_{m=1}^n \lambda_{n,m}^{-2} \sum_{k=0}^{2^{n-m}-1}\E\left(\sum_{i=k2^m+1}^{k2^m+2^{m-1}}\left(X_ {2^{m-1},i}-\E X_ {2^{m-1},i}\right)\right)^2\\
		&\le  \varepsilon^{-2}\sum_{m=1}^n 2^n \lambda_{n,m}^{-2}\E\left(X_{1}^2\mathbf{1} (|X_1|\le b_{2^{m-1}})\right).
	\end{split}
\end{equation*}
Using a similar estimate for the second term on the right hand side of \eqref{23}, we will finally obtain the following bound for the tail probability in \eqref{21}:
\begin{equation}\label{28}
	\mathbb{P}\left(\max_{1\le j <2^n}\left|S_{n,j}\right|>\varepsilon b_{2^n}\right) \le  C \varepsilon^{-2}\sum_{m=1}^n 2^n \lambda_{n,m}^{-2}\E\left(X_{1}^2\mathbf{1} (|X_1|\le b_{2^{m-1}})\right)+\text{ a neligible term}.
\end{equation}
Inequality \eqref{28} will play the role of the Kolmogorov maximal inequality in proving the Baum--Katz theorem (for the case where $1\le p<2$ and $\alpha=1/p$). We refer to Rio \cite{rio1995vitesses}
and Th\`{a}nh \cite{thanh2020theBaum,thanh2023extension,thanh2023hsu} for detailed arguments.

In this section, we use Rio's technique to establish some maximal inequalities
for double sums of dependent random variables. 
The following theorem presents a Rosenthal-type maximal inequality.
We would like to note that in Theorem \ref{thm.maximal.ineq1}, the underlying random variables are not necessary integrable.

\begin{theorem}\label{thm.maximal.ineq1}
	Let $\{X_{i,j}, \ i\ge 1,\ j\geq 1\}$ be a double array of nonnegative random variables satisfying Condition $(H_{2q})$
	for some $q\ge 1$, let $\{b_n,n\ge1\}$ be an increasing sequence of positive constants and let
	$\{\lambda_{m,n,i,j},1\le i\le m,1\le j\le n,m\ge1, n\ge1\}$ be an array of positive constants.
	For $s\ge 0,\ m\ge 1, n\ge 1$, set
	\begin{equation*}
		a_{m,n}=\sum_{i=1}^{m}\sum_{j=1}^{n}\lambda_{m,n,i,j}\ \text{ and }\ 
		X_{s,m,n}=X_{m,n}\mathbf{1}\left(X_{m,n}\le b_{2^{s}}\right)+b_{2^{s}}\mathbf{1}\left(X_{m,n}> b_{2^{s}}\right).	
	\end{equation*}
	Then for all $m\ge 1,n\ge1$,
	\begin{equation}\label{max32}
		\begin{split}
			&\E\left(\max_{\substack{1\le u< 2^m\\ 1\le v< 2^n}}\left|\sum_{i=1}^{u}\sum_{j=1}^{v}(X_{m+n,i,j}-\E X_{m+n,i,j})\right|^{2q}\right) \le C(q)\left(\sum_{s=1}^{m}\sum_{t=1}^{n}2^{s+t}b_{2^{s+t}}\max_{\substack{1\le i< 2^{m}\\ 1\le j< 2^{n}}}\mathbb{P}\left(X_{i,j}>b_{2^{s+t-2}}\right)\right)^{2q}\\
			&\qquad\qquad\qquad+C(q) a_{m,n}^{2q} \sum_{s=1}^m\sum_{t=1}^n 2^{m+n} \lambda_{m,n,s,t}^{-2q} \left(\max_{\substack{1\le i<2^{m}\\ 1\le j<2^{n}}}\E X_{s+t,i,j}^{2q}+2^{(s+t)(q-1)}\max_{\substack{1\le i<2^{m}\\ 1\le j<2^{n}}}\left(\E X_{s+t,i,j}^{2}\right)^{q}\right).
		\end{split}
	\end{equation}
\end{theorem}

The next theorem is a maximal inequality for double sums of dependent integrable random variables.

\begin{theorem}\label{thm.maximal.ineq2}
	Let the assumptions of Theorem \ref{thm.maximal.ineq1} be satisfied. Assume further that the random variables $X_{m,n}$, $m\ge1,n\ge1$ are
	integrable. Then for all $\varepsilon>0$, $m\ge 1$, $n\ge1$, we have
	\begin{equation}\label{max21}
		\begin{split}
			&\mathbb{P}\left(\max_{\substack{1\le u< 2^m\\ 1\le v< 2^n}}\left|\sum_{i=1}^{u}\sum_{j=1}^{v}(X_{i,j}-\E X_{i,j})\right|\ge 3a_{m,n}\varepsilon\right)	\le \sum_{i=1}^{2^{m}}\sum_{j=1}^{2^{n}}\mathbb{P}\left(X_{i,j}>b_{2^{m+n}}\right)\\
			&\qquad\qquad\qquad+C(q)\varepsilon^{-2q}\sum_{s=1}^m\sum_{t=1}^n 2^{m+n} \lambda_{m,n,s,t}^{-2q} \left(\max_{\substack{1\le i<2^{m}\\ 1\le j<2^{n}}}\E X_{s+t,i,j}^{2q}+2^{(s+t)(q-1)}\max_{\substack{1\le i<2^{m}\\ 1\le j<2^{n}}}\left(\E X_{s+t,i,j}^{2}\right)^{q}\right),
		\end{split}
	\end{equation}
	provided
	\begin{equation}\label{max22}
		\sum_{i=1}^{2^{m}}\sum_{j=1}^{2^{n}}\E\left(X_{i,j}\mathbf{1}\left(X_{i,j}>b_{2^{m+n}}\right)\right)\le \varepsilon a_{m,n}
	\end{equation}
	and
	\begin{equation}\label{max23}
		6\sum_{s=1}^{m}\sum_{t=1}^{n}2^{s+t}b_{2^{s+t}}\max_{\substack{1\le i< 2^{m}\\ 1\le j< 2^{n}}}\mathbb{P}\left(X_{i,j}>b_{2^{s+t-2}}\right)\le \varepsilon a_{m,n}.
	\end{equation}
\end{theorem}

\begin{remark}
	Before presenting the proof of Theorems \ref{thm.maximal.ineq1} and \ref{thm.maximal.ineq2}, we would like to provide some comments on these results. The maximal inequality \eqref{max32} can be regarded as a Rosenthal-type maximal inequality for double sums of truncated random variables. Theorem \ref{thm.maximal.ineq2} may be compared to Theorem 1.2 of Shao \cite{shao1995maximal}, which establishes a Rosenthal-type maximal inequality for $\rho$-mixing sequences. In proving the laws of large numbers for the maximum of the partial sums, we first choose $\lambda_{m,n,i,j}$ such that $a_{2^{m},2^{n}}\asymp b_{2^{m+n}}$. Then, under some moment conditions, \eqref{max22} and \eqref{max23} are satisfied, and the first term on the right-hand side of \eqref{max21} can be shown to be negligible. We are left with the last term on the right-hand side of \eqref{max21}, which can be controlled using moment calculations as in the usual proofs of laws of large numbers.
\end{remark}
\begin{proof}[Proof of Theorem \ref{thm.maximal.ineq1}]
	For $m,n,i,j\ge 1$, set
	\[\begin{split}
		X_{n,i,j}^{*}&=X_{n,i,j}-X_{n-1,i,j},\\
		Y_{n,i,j}^{*}&=X_{n,i,j}^{*}-\E X_{n,i,j}^{*},\\
		S_{k,\ell,i,j}&=\sum_{u=1}^{i}\sum_{v=1}^{j}\left(X_{k+\ell,u,v}-\E X_{k+\ell,u,v}\right), S_{k,\ell,0,j}=S_{k,\ell,i,0}=0,\ k\ge0,\ell\ge0,\\
		R_1(m,n)&=\sum_{s=1}^m\sum_{t=1}^n \max_{\substack{0\le k< 2^{m-s}\\ 0\le \ell< 2^{n-t}}}\left|\sum_{i=k2^s+1}^{k2^s+2^{s-1}}\sum_{j=\ell 2^t+1}^{\ell 2^t+2^{t-1}}(X_{s+t-2,i,j}-\E X_{s+t-2,i,j})\right|,\\
		R_2(m,n)&=\sum_{s=1}^{m}\sum_{t=1}^n \max_{\substack{0\le k< 2^{m-s}\\ 0\le \ell< 2^{n-t}}} \left|\sum_{i= k2^s+1}^{k2^s+2^{s-1}} \sum_{j=\ell 2^t+1}^{\ell 2^t+2^t} Y_{s+t-1,i,j}^{*}\right|,\\		
		R_3(m,n)&=\sum_{s=1}^{m}\sum_{t=1}^n \max_{\substack{0\le k< 2^{m-s}\\ 0\le \ell< 2^{n-t}}}\left| \sum_{i= k2^s+1}^{k2^s+2^{s}} \sum_{j=\ell 2^t+1}^{\ell 2^t+2^{t-1}} Y_{s+t-1,i,j}^{*}\right|,\\
		R_4(m,n)&=\sum_{s=1}^{m}\sum_{t=1}^n \max_{\substack{0\le k< 2^{m-s}\\ 0\le \ell< 2^{n-t}}} \left|\sum_{i= k2^s+1}^{k2^s+2^{s}} \sum_{j=\ell 2^t+1}^{\ell 2^t+2^{t}} \left(Y_{s+t,i,j}^{*}+Y_{s+t-1,i,j}^{*}\right)\right|.
	\end{split}\]
	Since the sequence $\{b_{n},n\ge1\}$ is increasing and $X_{i,j}$ are nonnegative,
	\begin{equation}\label{max14}
		0\le X_{n,i,j}^{*} \le b_{2^{n}}\mathbf{1}(X_{i,j}>b_{2^{n-1}}),\ n,i,j\ge 1.
	\end{equation}
	For simplicity, we write $R_i$ for $R_i(m,n)$. We will need the following claim whose proof is postponed to the Appendix.
	\begin{claim}\label{claim.estimate}
		For all $m\ge 1,n\ge1$,
		\begin{equation}\label{mz30}
			\begin{split}
				\max_{\substack{1\le u< 2^m\\ 1\le v< 2^n}}\left|S_{m,n,u,v}\right| &\le \sum_{i=1}^{4}R_{i}+6\sum_{s=1}^{m}\sum_{t=1}^{n}2^{s+t}b_{2^{s+t}}\max_{\substack{1\le i< 2^{m}\\ 1\le j< 2^{n}}}\mathbb{P}\left(X_{i,j}>b_{2^{s+t-2}}\right).
			\end{split}
		\end{equation}
	\end{claim}
	Now, we return to the proof of the theorem. 
	For all real numbers $x_1,\ldots,x_n$, we have the following elementary inequality:
	\begin{equation}\label{max33}
		|x_1+\cdots+x_n|^{2q}\le n^{2q-1}(|x_1|^{2q}+\cdots+|x_n|^{2q}).
	\end{equation}
	By using \eqref{mz30} and \eqref{max33}, we obtain
	\begin{equation}\label{max34}
		\E\left(\max_{\substack{1\le s< 2^m\\ 1\le t< 2^n}} |S_{m,n,s,t}|^{2q}\right)\le C(q)\left(\sum_{i=1}^4\E R_{i}^{2q}+\left(\sum_{s=1}^{m}\sum_{t=1}^{n}2^{s+t}b_{2^{s+t}}\max_{\substack{1\le i< 2^{m}\\ 1\le j< 2^{n}}}\mathbb{P}\left(X_{i,j}>b_{2^{s+t-2}}\right)\right)^{2q}\right).
	\end{equation}
	For $m\ge1,n\ge1$, set
	\[\lambda_{m,n}=\sum_{s=1}^{m}\sum_{t=1}^{n}\lambda_{m,n,s,t}^{2q/(2q-1)}.\]
	Then
	\begin{equation}\label{max35}
		\begin{split}
			\E R_{1}^{2q}&= \E\left(\sum_{s=1}^m\sum_{t=1}^n \lambda_{m,n,s,t}\left(\lambda_{m,n,s,t}^{-1}\max_{\substack{0\le k< 2^{m-s}\\ 0\le \ell< 2^{n-t}}}\left|\sum_{i=k2^s+1}^{k2^s+2^{s-1}}\sum_{j=\ell 2^t+1}^{\ell 2^t+2^{t-1}}(X_{s+t-2,i,j}-\E X_{s+t-2,i,j})\right|\right)\right)^{2q}\\
			&\le \lambda_{m,n}^{2q-1}\sum_{s=1}^m\sum_{t=1}^n \lambda_{m,n,s,t}^{-2q}\E\left(\max_{\substack{0\le k< 2^{m-s}\\ 0\le \ell< 2^{n-t}}}\left|\sum_{i=k2^s+1}^{k2^s+2^{s-1}}\sum_{j=\ell 2^t+1}^{\ell 2^t+2^{t-1}}(X_{s+t-2,i,j}-\E X_{s+t-2,i,j})\right|^{2q}\right)\\
			&\le \lambda_{m,n}^{2q-1} \sum_{s=1}^m\sum_{t=1}^n \lambda_{m,n,s,t}^{-2q} \sum_{k=0}^{2^{m-s}-1}\sum_{\ell=0}^{2^{n-t}-1}\E\left|\sum_{i=k2^s+1}^{k2^s+2^{s-1}}\sum_{j=\ell 2^t+1}^{\ell 2^t+2^{t-1}}(X_{s+t-2,i,j}-\E X_{s+t-2,i,j})\right|^{2q}\\
			&\le C(q)\lambda_{m,n}^{2q-1} \sum_{s=1}^m\sum_{t=1}^n \lambda_{m,n,s,t}^{-2q} 2^{m+n-s-t}\left(2^{s+t}\max_{\substack{1\le i<2^{m}\\ 1\le j<2^{n}}}\E X_{s+t-2,i,j}^{2q}+2^{(s+t)q}\max_{\substack{1\le i<2^{m}\\ 1\le j<2^{n}}}\left(\E X_{s+t-2,i,j}^{2}\right)^{q}\right)\\
			&= C(q)\lambda_{m,n}^{2q-1} \sum_{s=1}^m\sum_{t=1}^n 2^{m+n} \lambda_{m,n,s,t}^{-2q} \left(\max_{\substack{1\le i<2^{m}\\ 1\le j<2^{n}}}\E X_{s+t-2,i,j}^{2q}+2^{(s+t)(q-1)}\max_{\substack{1\le i<2^{m}\\ 1\le j<2^{n}}}\left(\E X_{s+t-2,i,j}^{2}\right)^{q}\right)\\
			&\le C(q)\lambda_{m,n}^{2q-1} \sum_{s=1}^m\sum_{t=1}^n 2^{m+n} \lambda_{m,n,s,t}^{-2q} \left(\max_{\substack{1\le i<2^{m}\\ 1\le j<2^{n}}}\E X_{s+t,i,j}^{2q}+2^{(s+t)(q-1)}\max_{\substack{1\le i<2^{m}\\ 1\le j<2^{n}}}\left(\E X_{s+t,i,j}^{2}\right)^{q}\right),
		\end{split}
	\end{equation}
	where we have applied H\"{o}lder's inequality in the first inequality, and Condition $(H_{2q})$ in the third inequality.
	The last inequality follows from the fact that $X_{s+t,i,j}\ge X_{s+t-2,i,j}\ge 0$. Now, for nonnegative real numbers
	$a_1,\ldots,a_n$, we have the following elementary inequality:
	\begin{equation}\label{elementary.ineq1}
		\sum_{i=1}^n a_{i}^{r}\le \left(\sum_{i=1}^{n}a_{i}\right)^{r},\ r\ge 1.
	\end{equation}
	Applying \eqref{elementary.ineq1}, we obtain
	\begin{equation}\label{eq.rem11}
		\lambda_{m,n}^{2q-1}=\left(\sum_{s=1}^{m}\sum_{t=1}^{n}\lambda_{m,n,s,t}^{2q/(2q-1)}\right)^{2q-1}\le \left(\sum_{s=1}^{m}\sum_{t=1}^{n}\lambda_{m,n,s,t}\right)^{2q}=a_{m,n}^{2q}.
	\end{equation}
	It follows from \eqref{max35} and \eqref{eq.rem11} that
	\begin{equation}\label{max36}
		\begin{split}
			\E R_{1}^{2q}&\le C(q)a_{m,n}^{2q} \sum_{s=1}^m\sum_{t=1}^n 2^{m+n} \lambda_{m,n,s,t}^{-2q} \left(\max_{\substack{1\le i<2^{m}\\ 1\le j<2^{n}}}\E X_{s+t,i,j}^{2q}+2^{(s+t)(q-1)}\max_{\substack{1\le i<2^{m}\\ 1\le j<2^{n}}}\left(\E X_{s+t,i,j}^{2}\right)^{q}\right).
		\end{split}
	\end{equation}
	By proceeding in the same manner as \eqref{max36} with noting that $0\le X_{n,i,j}^{*}\le X_{n,i,j}$, we have
	\begin{equation}\label{max37}
		\begin{split}
			\sum_{i=2}^{4}\mathbb{E} R_{i}^{2q}\le C(q) a_{m,n}^{2q} \sum_{s=1}^m\sum_{t=1}^n 2^{m+n} \lambda_{m,n,s,t}^{-2q} \left(\max_{\substack{1\le i<2^{m}\\ 1\le j<2^{n}}}\E X_{s+t,i,j}^{2q}+2^{(s+t)(q-1)}\max_{\substack{1\le i<2^{m}\\ 1\le j<2^{n}}}\left(\E X_{s+t,i,j}^{2}\right)^{q}\right).
		\end{split}
	\end{equation}
	Combining \eqref{max34}, \eqref{max36} and \eqref{max37} yields \eqref{max32}.	
\end{proof}

\begin{proof}[Proof of Theorem \ref{thm.maximal.ineq2}]
	We use the notations in the proof of Theorem \ref{thm.maximal.ineq1}. Let $\varepsilon>0$ be arbitrary. By using \eqref{max22}, \eqref{max23} and \eqref{mz30}, we have
	\begin{equation}\label{max24}
		\begin{split}
			&\mathbb{P}\left(\max_{\substack{1\le s< 2^m\\ 1\le t< 2^n}}\left|\sum_{i=1}^{s}\sum_{j=1}^{t}(X_{i,j}-\E X_{i,j})\right|\ge 3a_{m,n}\varepsilon\right) \le \sum_{i=1}^{2^{m}}\sum_{j=1}^{2^{n}}\mathbb{P}\left(X_{i,j}>b_{2^{m+n}}\right) \\
			&\quad+\mathbb{P}\left(\max_{\substack{1\le s< 2^m\\ 1\le t< 2^n}}\left|\sum_{i=1}^{s}\sum_{j=1}^{t}(X_{m+n,i,j}-\E X_{m+n,i,j})\right|+\sum_{i=1}^{2^{m}}\sum_{j=1}^{2^{n}}\E\left(X_{i,j}\mathbf{1}\left(X_{i,j}>b_{2^{m+n}}\right)\right)\ge 3a_{m,n}\varepsilon\right)\\
			&\quad	\le \sum_{i=1}^{2^{m}}\sum_{j=1}^{2^{n}}\mathbb{P}\left(X_{i,j}>b_{2^{m+n}}\right)+
			\mathbb{P}\left(\sum_{i=1}^{4}R_{i}+6\sum_{s=1}^{m}\sum_{t=1}^{n}2^{s+t}b_{2^{s+t}}\max_{\substack{1\le i< 2^{m}\\ 1\le j< 2^{n}}}\mathbb{P}\left(X_{i,j}>b_{2^{s+t-2}}\right)\ge 2a_{m,n}\varepsilon\right)\\
			&\quad	\le \sum_{i=1}^{2^{m}}\sum_{j=1}^{2^{n}}\mathbb{P}\left(X_{i,j}>b_{2^{m+n}}\right)+\mathbb{P}\left(\sum_{i=1}^{4}R_{i}\ge a_{m,n}\varepsilon\right).
		\end{split}
	\end{equation}
	By applying Markov's inequality, \eqref{max33}, \eqref{max36} and \eqref{max37}, we have
	\begin{equation}\label{max25}
		\begin{split}
			&\mathbb{P}\left(\sum_{i=1}^{4}R_{i}\ge a_{m,n}\varepsilon\right)
			\le 
			\varepsilon^{-2q}a_{m,n}^{-2q}\E\left(\sum_{i=1}^{4}R_{i}\right)^{2q}\\
			&\qquad\qquad\qquad \le 4^{2q-1}\varepsilon^{-2q}a_{m,n}^{-2q}\sum_{i=1}^{4}\E R_{i}^{2q}\\
			&\qquad\qquad\qquad\le C(q)\varepsilon^{-2q}\sum_{s=1}^m\sum_{t=1}^n 2^{m+n} \lambda_{m,n,s,t}^{-2q} \left(\max_{\substack{1\le i<2^{m}\\ 1\le j<2^{n}}}\E X_{s+t,i,j}^{2q}+2^{(s+t)(q-1)}\max_{\substack{1\le i<2^{m}\\ 1\le j<2^{n}}}\left(\E X_{s+t,i,j}^{2}\right)^{q}\right).
		\end{split}
	\end{equation}
	Combining \eqref{max24} and \eqref{max25} yields \eqref{max21}.
\end{proof}

In the following corollary, we do not assume the underlying random variables are nonnegative.
The proof is done by using the identity $X=X^{+}-X^{-}$ for every random variable $X$
and then applying Theorems \ref{thm.maximal.ineq1} and \ref{thm.maximal.ineq2}.
We omit the details.

\begin{corollary}\label{cor.maximal.ineq5}
	Let $\{X_{i,j}, \ i\ge 1,\ j\geq 1\}$ be a double array of integrable random variables satisfying Condition $(H_{2q})$
	for some $q\ge 1$, let $\{b_n,n\ge1\}$ be an increasing sequence of positive constants and let
	$\{\lambda_{m,n,i,j},1\le i\le m,1\le j\le n,m\ge1, n\ge1\}$ be an array of positive constants.
	For $s\ge 0,\ m\ge 1, n\ge 1$, set
	\begin{equation*}
		a_{m,n}=\sum_{s=1}^{m}\sum_{t=1}^{n}\lambda_{m,n,s,t}\ \text{ and }\ 
		X_{s,m,n}=-b_{2^{s}}\mathbf{1}\left(X_{m,n}<-b_{2^{s}}\right)+X_{m,n}\mathbf{1}\left(|X_{m,n}|\le b_{2^{s}}\right)+b_{2^{s}}\mathbf{1}\left(X_{m,n}> b_{2^{s}}\right).	
	\end{equation*}
	Then for all $m\ge 1,n\ge1$ and for all $\varepsilon>0$, the following two inequalities hold:
	\begin{description}
		\item[(i)]	\begin{equation*}
			\begin{split}
				&\E\left(\max_{\substack{1\le u< 2^m\\ 1\le v< 2^n}}\left|\sum_{i=1}^{u}\sum_{j=1}^{v}(X_{m+n,i,j}-\E X_{m+n,i,j})\right|^{2q}\right) \le C(q)\left(\sum_{s=1}^{m}\sum_{t=1}^{n}2^{s+t}b_{2^{s+t}}\max_{\substack{1\le i< 2^{m}\\ 1\le j< 2^{n}}}\mathbb{P}\left(|X_{i,j}|>b_{2^{s+t-2}}\right)\right)^{2q}\\
				&\qquad\qquad+C(q) a_{m,n}^{2q} \sum_{s=1}^m\sum_{t=1}^n 2^{m+n} \lambda_{m,n,s,t}^{-2q} \left(\max_{\substack{1\le i<2^{m}\\ 1\le j<2^{n}}}\E|X_{s+t,i,j}|^{2q}+2^{(s+t)(q-1)}\max_{\substack{1\le i<2^{m}\\ 1\le j<2^{n}}}\left(\E X_{s+t,i,j}^{2}\right)^{q}\right).
			\end{split}
		\end{equation*}
		\item[(ii)]	\begin{equation*}
			\begin{split}
				&\mathbb{P}\left(\max_{\substack{1\le s< 2^m\\ 1\le t< 2^n}}\left|\sum_{i=1}^{s}\sum_{j=1}^{t}(X_{i,j}-\E X_{i,j})\right|\ge 6a_{m,n}\varepsilon\right)	\le \sum_{i=1}^{2^{m}}\sum_{j=1}^{2^{n}}\mathbb{P}\left(|X_{i,j}|>b_{2^{m+n}}\right)\\
				&\qquad\qquad+C(q)\varepsilon^{-2q}\sum_{s=1}^m\sum_{t=1}^n 2^{m+n} \lambda_{m,n,s,t}^{-2q} \left(\max_{\substack{1\le i<2^{m}\\ 1\le j<2^{n}}}\E |X_{s+t,i,j}|^{2q}+2^{(s+t)(q-1)}\max_{\substack{1\le i<2^{m}\\ 1\le j<2^{n}}}\left(\E X_{s+t,i,j}^{2}\right)^{q}\right),
			\end{split}
		\end{equation*}
		provided
		\begin{equation*}
			\sum_{i=1}^{2^{m}}\sum_{j=1}^{2^{n}}\E\left(|X_{i,j}|\mathbf{1}\left(|X_{i,j}|>b_{2^{m+n}}\right)\right)\le \varepsilon a_{m,n}
		\end{equation*}
		and
		\begin{equation*}
		6\sum_{s=1}^{m}\sum_{t=1}^{n}2^{s+t}b_{2^{s+t}}\max_{\substack{1\le i< 2^{m}\\ 1\le j< 2^{n}}}\mathbb{P}\left(|X_{i,j}|>b_{2^{s+t-2}}\right)\le \varepsilon a_{m,n}.
		\end{equation*}
	\end{description}
\end{corollary}

\section{The proof of the Hsu--Robbins--Erd\"{o}s--Spitzer--Baum--Katz theorem for dependent random fields}\label{sec.bk}

In this section, we will prove Theorem \ref{thm.BKHRE}.
The proof is based on a Rosenthal-type maximal inequality in Theorem \ref{thm.maximal.ineq2}.

\begin{proof}[Proof of Theorem \ref{thm.BKHRE}]
	Firstly, we prove the sufficiency part. Since the arrays $\{X_{m,n}^{+}, \ m\ge 1,\ n \geq 1\}$ and $\{X_{m,n}^{-}, \ m\ge 1,\ n \geq 1\}$ also satisfy the
	assumptions of the theorem, we can assume, without loss of generality, that $X_{m,n}\ge 0$ for all $m\ge1,\ n\ge 1$.
	For $s\ge 0,\ m\ge 1, n\ge 1$, set
	\begin{equation*}
		b_{n}=n^\alpha \ \text{ and } \	X_{s,m,n}=X_{m,n}\mathbf{1}\left(X_{m,n}\le b_{2^{s}}\right)+b_{2^{s}}\mathbf{1}\left(X_{m,n}> b_{2^{s}}\right).
	\end{equation*}
	For $m\ge1,n\ge1,1\le s\le m,1\le t\le n$, let
	\[\frac{\alpha p}{2q}<a<\alpha,\ \lambda_{m,n,s,t}=2^{a(m+n)+(\alpha-a)(s+t)},\]
	and
	\[a_{m,n}=\sum_{s=1}^{m}\sum_{t=1}^{n}\lambda_{m,n,s,t}.\]
	Then 
	\begin{equation}\label{lambda.mn.1}
		\begin{split}
			b_{2^{m+n}}&=\lambda_{m,n,m,n}\\
			&\le a_{m,n}=2^{a(m+n)}\sum_{s=1}^m\sum_{t=1}^n2^{(\alpha-a)(s+t)}\\
			&\le C_1(a,\alpha) b_{2^{m+n}},\ m\ge1,n\ge1.
		\end{split}
	\end{equation}
	Let $\varepsilon>0$ be arbitrary. The proof of \eqref{mz05} will be completed if we can show that
	\begin{equation}\label{mz11}
		\sum_{m=1}^{\infty}\sum_{n=1}^{\infty}
		2^{(m+n)(\alpha p-1)}\mathbb{P}\left(\max_{\substack{1\le u< 2^m\\ 1\le v< 2^n}}\left|
		\sum_{i=1}^{u}\sum_{j=1}^{v}(X_{i,j}-\E X_{i,j})\right|>3C_1(a,\alpha)\varepsilon b_{2^{m+n}}\right)<\infty.
	\end{equation}
By using \eqref{mz04} and the Lebesgue dominated convergence theorem, we have
	\begin{equation}\label{mz12}
		\begin{split}
\lim_{x\to\infty}\E\left(|X|^{p}\mathbf{1}\left(|X|>x\right)\right)=0.
		\end{split}
	\end{equation}
Since $\alpha p\ge 1$,	it follows from the stochastic domination assumption, the first inequality in \eqref{lambda.mn.1}, and \eqref{mz12} that
	\begin{equation}\label{mz13}
		\begin{split}
			\lim_{m\vee n\to\infty}\frac{\sum_{i=1}^{2^{m}}\sum_{j=1}^{2^{n}}\E\left(X_{i,j}\mathbf{1}\left(X_{i,j}>b_{2^{m+n}}\right)\right)}{a_{m,n}}&\le \lim_{m\vee n\to\infty}\frac{2^{m+n}\E\left(|X|\mathbf{1}\left(|X|>2^{(m+n)\alpha}\right)\right)}{2^{(m+n)\alpha}}\\
			&\le \lim_{m\vee n\to\infty}\frac{\E\left(|X|^{p}\mathbf{1}\left(|X|>2^{(m+n)\alpha}\right)\right)}{2^{(m+n)(\alpha p-1)}}=0.
		\end{split}
	\end{equation}
	It is clear that \eqref{mz04} implies $\lim_{n\to\infty}n\mathbb{P}(|X|>n^{\alpha})\le \lim_{n\to\infty}n\mathbb{P}(|X|>n^{1/p})=0$. It thus follows from the stochastic domination
	and a double sum analogue of the Toeplitz lemma (see Lemma 2.2 in \cite{stadtmuller2011strong}) that
	\begin{equation}\label{mz14}
		\begin{split}
			&\lim_{m\vee n\to\infty}\frac{\sum_{s=1}^{m}\sum_{t=1}^{n}2^{s+t}b_{2^{s+t}}\max_{1\le i< 2^{m}, 1\le j< 2^{n}}\mathbb{P}\left(X_{i,j}>b_{2^{s+t-2}}\right)}{a_{m,n}}\\
			&\le \lim_{m\vee n\to\infty}\frac{4\sum_{s=1}^{m}\sum_{t=1}^{n}2^{(s+t)\alpha}2^{s+t-2}\mathbb{P}\left(|X|>2^{(s+t-2)\alpha}\right)}{2^{(m+n)\alpha}}=0.
		\end{split}
	\end{equation}
	It follows from \eqref{mz13} and \eqref{mz14} that there exists $n_0$ such that \eqref{max22} and \eqref{max23} holds for all $m\vee n\ge n_0$. 
	We will now consider the following two cases.

	\noindent\textit{Case 1: $1\le p<2$.} In this case, the array $\{X_{m,n},m\ge1,n\ge1\}$ satisfies Condition $(H_{2q})$ with $q=1$.
	Applying \eqref{lambda.mn.1}, Theorem \ref{thm.maximal.ineq2} with $q=1$,
and the stochastic domination assumption, we have
	\begin{equation}\label{mz15}
		\begin{split}
			& \sum_{m\vee n\ge n_0}2^{(m+n)(\alpha p-1)}\mathbb{P}\left(\max_{\substack{1\le u< 2^m\\ 1\le v< 2^n}}\left|
			\sum_{i=1}^{u}\sum_{j=1}^{v}(X_{i,j}-\E X_{i,j})\right|>3C_1(a,\alpha)\varepsilon b_{2^{m+n}}\right)\\
			&\le \sum_{m\vee n\ge n_0}2^{(m+n)(\alpha p-1)}\mathbb{P}\left(\max_{\substack{1\le u< 2^m\\ 1\le v< 2^n}}\left|\sum_{i=1}^{u}\sum_{j=1}^{v}(X_{i,j}-\E X_{i,j})\right|\ge 3a_{m,n}\varepsilon\right)\\
			&\le \sum_{m\vee n\ge n_0}2^{(m+n)(\alpha p-1)}\sum_{i=1}^{2^{m}}\sum_{j=1}^{2^{n}}\mathbb{P}\left(X_{i,j}>b_{2^{m+n}}\right)\\
			&\quad+C\varepsilon^{-2}\sum_{m\vee n\ge n_0}2^{(m+n)(\alpha p-1)}\sum_{s=1}^m\sum_{t=1}^n 2^{m+n} \lambda_{m,n,s,t}^{-2} \max_{\substack{1\le i<2^{m}\\ 1\le j<2^{n}}}\E X_{s+t,i,j}^{2}\\
			&\le \sum_{m\vee n\ge n_0}2^{(m+n)\alpha p}
			\mathbb{P}\left(|X|>b_{2^{m+n}}\right)\\
			&\quad+C\varepsilon^{-2}\sum_{m\vee n\ge n_0}2^{(m+n)\alpha p}\sum_{s=1}^m\sum_{t=1}^n \lambda_{m,n,s,t}^{-2} \left(\E\left(X^{2}\mathbf{1}(|X|\le b_{2^{s+t}})\right)+b_{2^{s+t}}^{2}\mathbb{P}(|X|>b_{2^{s+t}})\right).
		\end{split}
	\end{equation}
	By using \eqref{mz04} and Lemma \ref{lem.moment.estimate03}, we have
	\begin{equation}\label{mz17}
		\begin{split}
			\sum_{m\vee n\ge n_0}2^{(m+n)\alpha p}
			\mathbb{P}\left(|X|>b_{2^{m+n}}\right)&=\sum_{m\vee n\ge n_0}2^{(m+n)\alpha p}
			\mathbb{P}\left(|X|>2^{(m+n)\alpha}\right)<\infty.
		\end{split}
	\end{equation}
	From \eqref{mz15} and \eqref{mz17}, the proof of \eqref{mz11} will be completed if we can show that
	\begin{equation}\label{mz72}
		\begin{split}
			I_1:=\sum_{m\vee n\ge n_0}2^{(m+n)\alpha p}\sum_{s=1}^m\sum_{t=1}^n \lambda_{m,n,s,t}^{-2}\E\left(X^{2}\mathbf{1}(|X|\le b_{2^{s+t}})\right)<\infty,
		\end{split}
	\end{equation}
	and
	\begin{equation}\label{mz73}
		\begin{split}
			I_2:=\sum_{m\vee n\ge n_0}2^{(m+n)\alpha p}\sum_{s=1}^m\sum_{t=1}^n \lambda_{m,n,s,t}^{-2}b_{2^{s+t}}^{2}\mathbb{P}(|X|>b_{2^{s+t}})<\infty.
		\end{split}
	\end{equation}
	Note that in the case where $1\le p<2$, we have $q=1$ and thus $\alpha p<2a$. Therefore, by using \eqref{mz04} and Lemma \ref{lem.moment.estimate03} again, we have
	\begin{equation*}
		\begin{split}
			I_1&\le \sum_{m=1}^{\infty}\sum_{n=1}^{\infty}2^{(m+n)(\alpha p-2a)} \sum_{s=1}^m\sum_{t=1}^n 2^{-2(\alpha-a)(s+t)} \E\left(X^2\mathbf{1}\left(|X|\le b_{2^{s+t}}\right)\right)\\
			&=\sum_{s=1}^{\infty}\sum_{t=1}^{\infty}\left(\sum_{m=s}^{\infty}\sum_{n=t}^{\infty}2^{(m+n)(\alpha p-2a)} \right)2^{-2(\alpha-a)(s+t)} \E\left(X^2\mathbf{1}\left(|X|\le b_{2^{s+t}}\right)\right)\\
			&=C(a,\alpha,p)\sum_{s=1}^{\infty}\sum_{t=1}^{\infty} 2^{(s+t)(\alpha p-2a)} 2^{-2(\alpha-a)(s+t)} \E\left(X^2\mathbf{1}\left(|X|\le b_{2^{s+t}}\right)\right)\\
			&=C(a,\alpha,p)\sum_{s=1}^{\infty}\sum_{t=1}^{\infty} 2^{(s+t)\alpha (p-2)} \E\left(X^2\mathbf{1}\left(|X|\le 2^{(s+t)\alpha}\right)\right)<\infty
		\end{split}
	\end{equation*}
	and
	\begin{equation*}
		\begin{split}
			I_2&\le \sum_{m=1}^{\infty}\sum_{n=1}^{\infty}2^{(m+n)(\alpha p-2a)} \sum_{s=1}^m\sum_{t=1}^n 2^{2a(s+t)} \mathbb{P}\left(|X|> b_{2^{s+t}}\right)\\
			&=\sum_{s=1}^{\infty}\sum_{t=1}^{\infty}\left(\sum_{m=s}^{\infty}\sum_{n=t}^{\infty}2^{(m+n)(\alpha p-2a)} \right)2^{2a(s+t)} \mathbb{P}\left(|X|> b_{2^{s+t}}\right)\\
			&= C(a,\alpha,p)\sum_{s=1}^{\infty}\sum_{t=1}^{\infty} 2^{(s+t)(\alpha p-2a)} 2^{2a(s+t)} \mathbb{P}\left(|X|> b_{2^{s+t}}\right)\\
			&= C(a,\alpha,p)\sum_{s=1}^{\infty}\sum_{t=1}^{\infty} 2^{(s+t)\alpha p} \mathbb{P}\left(|X|> 2^{(s+t)\alpha}\right)<\infty
		\end{split}
	\end{equation*}
	thereby proving \eqref{mz72} and \eqref{mz73}.

	\noindent\textit{Case 2: $p\ge 2$.} In this case, we have from \eqref{mz04} that $\E X^2<\infty$. 
	By applying \eqref{lambda.mn.1}, Theorem \ref{thm.maximal.ineq2}
	and the stochastic domination assumption, we have
	\begin{equation}\label{hr15}
		\begin{split}
			& \sum_{m\vee n\ge n_0}2^{(m+n)(\alpha p-1)}\mathbb{P}\left(\max_{\substack{1\le u< 2^m\\ 1\le v< 2^n}}\left|
			\sum_{i=1}^{u}\sum_{j=1}^{v}(X_{i,j}-\E X_{i,j})\right|>3C_1(a,\alpha)\varepsilon b_{2^{m+n}}\right)\\
			&\le \sum_{m\vee n\ge n_0}2^{(m+n)(\alpha p-1)}\mathbb{P}\left(\max_{\substack{1\le u< 2^m\\ 1\le v< 2^n}}\left|\sum_{i=1}^{u}\sum_{j=1}^{v}(X_{i,j}-\E X_{i,j})\right|\ge 3a_{m,n}\varepsilon\right)\\
			&\le \sum_{m\vee n\ge n_0}2^{(m+n)(\alpha p-1)}\sum_{i=1}^{2^{m}}\sum_{j=1}^{2^{n}}\mathbb{P}\left(X_{i,j}>b_{2^{m+n}}\right)\\
			&\quad+C(q)\varepsilon^{-2q}\sum_{m\vee n\ge n_0}2^{(m+n)\alpha p}\sum_{s=1}^m\sum_{t=1}^n \lambda_{m,n,s,t}^{-2q} \left(\max_{\substack{1\le i<2^{m}\\ 1\le j<2^{n}}}\E X_{s+t,i,j}^{2q}+2^{(s+t)(q-1)}\max_{\substack{1\le i<2^{m}\\ 1\le j<2^{n}}}\left(\E X_{s+t,i,j}^{2}\right)^{q}\right)\\
			&\le \sum_{m\vee n\ge n_0}2^{(m+n)\alpha p}
			\mathbb{P}\left(|X|>b_{2^{m+n}}\right)\\
			&+C\sum_{m\vee n\ge n_0}2^{(m+n)\alpha p}\sum_{s=1}^m\sum_{t=1}^n \lambda_{m,n,s,t}^{-2q} \left(\E|X|^{2q}\mathbf{1}(|X|\le b_{2^{s+t}})+b_{2^{s+t}}^{2q}\mathbb{P}(|X|>b_{2^{s+t}})+2^{(s+t)(q-1)}(\E X^2)^{q}\right).
		\end{split}
	\end{equation}
	By using \eqref{mz04} and Lemma \ref{lem.moment.estimate03} again, we have \eqref{mz17} still holds in this case.
	From \eqref{mz17}, \eqref{hr15} and the fact that $\E X^2<\infty$, the proof of \eqref{mz11} will be completed if we can show that
	\begin{equation}\label{hr72}
		\begin{split}
			J_1:=\sum_{m=1}^{\infty}\sum_{n=1}^{\infty}2^{(m+n)\alpha p}\sum_{s=1}^m\sum_{t=1}^n \lambda_{m,n,s,t}^{-2q}\E\left(|X|^{2q}\mathbf{1}\left(|X|\le b_{2^{s+t}}\right)\right)<\infty,
		\end{split}
	\end{equation}
	\begin{equation}\label{hr73}
		\begin{split}
			J_2:=\sum_{m=1}^{\infty}\sum_{n=1}^{\infty}2^{(m+n)\alpha p}\sum_{s=1}^m\sum_{t=1}^n \lambda_{m,n,s,t}^{-2q}b_{2^{s+t}}^{2q}\mathbb{P}\left(|X|> b_{2^{s+t}}\right)<\infty,
		\end{split}
	\end{equation}
	and
	\begin{equation}\label{hr74}
		\begin{split}
			J_3:=\sum_{m=1}^{\infty}\sum_{n=1}^{\infty}2^{(m+n)\alpha p}\sum_{s=1}^m\sum_{t=1}^n \lambda_{m,n,s,t}^{-2q}2^{(s+t)(q-1)}<\infty.
		\end{split}
	\end{equation}
	Since $q>(\alpha p-1)/(2\alpha-1)$, we have 
	\[\frac{\alpha p}{2q}<\alpha-\frac{q-1}{2q}<\alpha.\]
	Therefore, we can let $a$ be such that
	\begin{equation}\label{def.ab}
		\alpha-\frac{q-1}{2q}<a<\alpha.
	\end{equation}
	The proofs of \eqref{hr72} and \eqref{hr73} are the same as that of \eqref{mz72} and \eqref{mz73}, respectively. Finally, by using \eqref{def.ab}
	and noting again that $(2\alpha-1)q>\alpha p-1$, we have
	\begin{equation*}\label{hr76}
		\begin{split}
			J_3&=\sum_{m=1}^{\infty}\sum_{n=1}^{\infty}2^{(m+n)(\alpha p-2qa)}\sum_{s=1}^m\sum_{t=1}^n 2^{(s+t)(q-1-2q(\alpha -a))}\\
			&\le C\sum_{m=1}^{\infty}\sum_{n=1}^{\infty}2^{(m+n)(\alpha p+q-1-2\alpha q)}<\infty
		\end{split}
	\end{equation*}
	thereby proving \eqref{hr74}. The proof of the sufficiency part is completed.
	
	Now, we will prove the necessity part. Assume that \eqref{mz05.converse} holds. 
	Without loss of generality, we can assume that $\mu=0$. It is clear that this implies that for all $\varepsilon>0$,
	\begin{equation}\label{hr78}
		\sum_{m=1}^{\infty}\sum_{n=1}^{\infty}
		(mn)^{\alpha p-2}\mathbb{P}\left(\max_{1\le k\le m, 1\le \ell\le n}\left|X_{k,\ell}\right|>\varepsilon (mn)^{\alpha}\right)<\infty
	\end{equation}
	and that
	\begin{equation}\label{hr80}
		\lim_{m\vee n\to\infty}\mathbb{P}\left(\max_{1\le k\le m, 1\le \ell\le n}\left|X_{k,\ell}\right|>\varepsilon (mn)^{\alpha}\right)=0.
	\end{equation}
	Since $\{X_{m,n}, \ m\ge 1,\ n \geq 1\}$ satisfies Condition $(H_{2})$, we obtain from \eqref{hr80} and Lemma \ref{lem.BC}
	that
	\begin{equation}\label{mz13.converse}
		mn\mathbb{P}\left(|X_{1,1}|>(mn)^{\alpha}\right)\le C\mathbb{P}\left(\max_{1\le k\le m, 1\le \ell\le n} |X_{k,\ell}|>(mn)^{\alpha}\right)
	\end{equation}
	whenever $m\vee n\ge n_1$ for some positive integer $n_1$. Combining \eqref{hr78} and \eqref{mz13.converse}, we have
	\begin{equation}\label{mz14.converse}
		\begin{split}
			\sum_{m\vee n\ge n_1}(mn)^{\alpha p-1}\mathbb{P}\left(|X|>(mn)^{\alpha}\right)&=\sum_{m\vee n\ge n_1}(mn)^{\alpha p-1}\mathbb{P}\left(|X_{1,1}|>(mn)^{\alpha}\right)\\
			&\le C \sum_{m\vee n\ge n_1}(mn)^{\alpha p-2}\mathbb{P}\left(\max_{1\le k\le m, 1\le \ell\le n} |X_{k,\ell}|>(mn)^{\alpha}\right)\\
			&<\infty.
		\end{split}
	\end{equation}
	Applying Lemma \ref{lem.moment.estimate03}, we have from \eqref{mz14.converse} that $\E\left(|X|^p \log(|X|)\right)<\infty$ thereby establishing \eqref{mz04}.
	Since \eqref{mz04} holds, we can apply the sufficiency part to conclude that \eqref{mz05} holds. By using Remark \ref{rem:slln1}, we obtain from \eqref{mz05} that
	\begin{equation}\label{mz16.converse}
		\lim_{m\vee n\to\infty}\dfrac{\sum_{i=1}^{m}\sum_{j=1}^{n}(X_{i,j}-\E X_{i,j})}{(mn)^{\alpha}}=\lim_{m\vee n\to\infty}\left(\dfrac{\sum_{i=1}^{m}\sum_{j=1}^{n}X_{i,j}}{(mn)^{\alpha}}-(mn)^{1-\alpha}\E X\right)=0 \ \text{ a.s.}
	\end{equation}
	Similarly,  \eqref{mz05.converse} (with $\mu=0$) implies
	\begin{equation}\label{mz17.converse}
		\lim_{m\vee n\to\infty}\dfrac{\sum_{i=1}^{m}\sum_{j=1}^{n}X_{i,j}}{(mn)^{\alpha}}=0 \ \text{ a.s.}
	\end{equation}
	Since $\alpha\le 1$, we obtain from \eqref{mz16.converse} and \eqref{mz17.converse} that $\E X=0$ thereby completing the proof of the necessity part.
\end{proof}

%%%%%
\section{The proof of the Feller WLLN and the Pyke--Root theorem for dependent random fields}\label{sec.feller-pyke-root}

In this section, we present the proof of Theorems \ref{thm.Feller} and \ref{thm.PR}.
In these theorems, the underlying random variables are only required to satisfy Condition $(H_2)$.
Therefore, they can apply for pairwise independent random variables and pairwise negatively dependent random variables.

\begin{proof}[Proof of Theorem \ref{thm.Feller}]
	We first prove the sufficiency part. Assume that \eqref{fel03} holds. As in Section \ref{sec.bk}, it suffices to consider the case $X_{m,n}\ge 0$ for all $m\ge1,n\ge1$. Set
	\[X_{s,m,n}=X_{m,n}\mathbf{1}\left(X_{m,n}\le b_{2^{s}}\right)+b_{2^{s}}\mathbf{1}\left(X_{m,n}> b_{2^{s}}\right),\ s\ge 0,m\ge 1,n\ge1.\]
	For $m\ge1,\ n\ge1$, let $k\ge 1,\ell\ge 1$ be such that $2^{k-1}\le m<2^{k}$, $2^{\ell-1}\le n<2^{\ell}$.
	Then
	\begin{equation}\label{fel07}
		\begin{split}
			\frac{\max_{u\le m,v\le n}\left|\sum_{i=1}^{u}\sum_{j=1}^{v} (X_{i,j}-\E Z_{mn,i,j})\right|}{b_{mn}}&\le \frac{4\max_{u< 2^{k},v<2^{\ell}}\left|\sum_{i=1}^{u}\sum_{j=1}^{v} (X_{i,j}-X_{k+\ell,i,j})\right|}{b_{2^{k+\ell}}}\\
			&\qquad+\frac{4\max_{u< 2^{k},v<2^{\ell}}\left|\sum_{i=1}^{u}\sum_{j=1}^{v} \E(X_{k+\ell,i,j}-Z_{mn,i,j})\right|}{b_{2^{k+\ell}}}\\
			&\qquad+\frac{4\max_{u< 2^{k},v<2^{\ell}}\left|\sum_{i=1}^{u}\sum_{j=1}^{v} (X_{k+\ell,i,j}-\E X_{k+\ell,i,j})\right|}{b_{2^{k+\ell}}}\\
			&:=4\left(K_{1}(k,\ell)+K_{2}(k,\ell,m,n)+K_{3}(k,\ell)\right).
		\end{split}
	\end{equation}
	The rest of the proof of the sufficiency part will be divided into three steps. 
	
	\noindent\textbf{Step 1:} Prove 
	\begin{equation}\label{fel08}
		\lim_{k\vee \ell\to\infty}K_{1}(k,\ell)=0\ \text{ in probability}.
	\end{equation}
	Let $\varepsilon>0$ be arbitrary. By \eqref{fel03}, we have
	\begin{equation*}
		\begin{split}
			\mathbb{P}\left(K_{1}(k,\ell)>\varepsilon\right)& \le \mathbb{P}\left(\bigcup_{i=1}^{2^{k}}\bigcup_{i=1}^{2^{\ell}}(X_{i,j}\not=X_{k+\ell,i,j})\right)\\
			&\le \sum_{i=1}^{2^{k}}\sum_{j=1}^{2^{\ell}} \mathbb{P}\left(X_{i,j}>b_{2^{k+\ell}}\right)\\
			&\le 2^{k+\ell}\mathbb{P}\left(|X|>b_{2^{k+\ell}}\right)\to 0 \ \text{ as }\ k\vee\ell\to\infty
		\end{split}
	\end{equation*}
	thereby establishing \eqref{fel08}.
	
	\noindent\textbf{Step 2:} Prove 
	\begin{equation}\label{fel09}
		\lim_{k\vee \ell\to\infty} \max_{2^{k-1}\le m<2^k,2^{\ell-1}\le n<2^\ell}K_{2}(k,\ell,m,n)=0.
	\end{equation}
	For all $2^{k-1}\le m<2^{k}$, $2^{\ell-1}\le n<2^{\ell}$, it is clear that
	\begin{equation*}
		\begin{split}
			0&\le X_{k+\ell,i,j}-Z_{mn,i,j}\\
			&\le X_{i,j}\mathbf{1}(b_{2^{k+\ell-2}}<X_{i,j}\le b_{2^{k+\ell}})+b_{2^{k+\ell}}\mathbf{1}(X_{i,j}>b_{2^{k+\ell}})\\
			&\le b_{2^{k+\ell}}\mathbf{1}(X_{i,j}>b_{2^{k+\ell-2}}).
		\end{split}
	\end{equation*}
	It thus follows from the stochastic domination assumption and \eqref{fel03} that
	\begin{equation*}
		\begin{split}
			\max_{2^{k-1}\le m<2^k,2^{\ell-1}\le n<2^\ell}K_{2}(k,\ell,m,n)&\le \frac{\sum_{i=1}^{2^{k}}\sum_{j=1}^{2^{\ell}}b_{2^{k+\ell}}\mathbb{P}(X_{i,j}>b_{2^{k+\ell-2}})}{b_{2^{k+\ell}}}\\
			&\le 2^{k+\ell}\mathbb{P}(|X|>b_{2^{k+\ell-2}})\to 0 \ \text{ as }\ k\vee \ell\to\infty
		\end{split}
	\end{equation*}
	thereby establishing \eqref{fel09}.
	
	\noindent\textbf{Step 3:} Prove 
	\begin{equation}\label{fel11}
		\lim_{k\vee \ell\to\infty}K_{3}(k,\ell)=0\ \text{ in probability}.
	\end{equation}
	This is the most difficult part.
	For $m\ge1,n\ge1,1\le s\le m,1\le t\le n$, set
	\[1/2<a<1/p,\ \lambda_{m,n,s,t}=2^{a(m+n)+(1/p-a)(s+t)},\]
	and
	\[a_{m,n}=\sum_{s=1}^{m}\sum_{t=1}^{n}\lambda_{m,n,s,t}.\]
	Then, similar to \eqref{lambda.mn.1}, we have 
	\begin{equation}\label{lambda.mn.2}
		\begin{split}
			b_{2^{m+n}}	&\le a_{m,n}\le C_1(a,p) b_{2^{m+n}},\ m\ge1,n\ge1.
		\end{split}
	\end{equation}
	By using the second inequality in \eqref{lambda.mn.2} and Theorem \ref{thm.maximal.ineq1} with $q=1$, we have 
	\begin{equation}\label{fel12}
	\begin{split}
			0&\le b_{2^{k+\ell}}^{-2}\E\left(\max_{\substack{1\le u< 2^{k}\\ 1\le v< 2^{\ell}}}\left(\sum_{i=1}^{v}\sum_{j=1}^{v}(X_{k+\ell,i,j}-\E X_{k+\ell,i,j})\right)^{2}\right)\\
			&\le C_1(a,p)^{2} a_{k,\ell}^{-2}\E\left(\max_{\substack{1\le u< 2^{k}\\ 1\le v< 2^{\ell}}}\left(\sum_{i=1}^{u}\sum_{j=1}^{v}(X_{k+\ell,i,j}-\E X_{k+\ell,i,j})\right)^{2}\right)\\
			& \le C\left(\left(a_{k,\ell}^{-1}\sum_{s=1}^{k}\sum_{t=1}^{\ell}2^{s+t}b_{2^{s+t}}\max_{\substack{1\le i< 2^{k}\\ 1\le j< 2^{\ell}}}\mathbb{P}\left(X_{i,j}>b_{2^{s+t-2}}\right)\right)^{2}+\sum_{s=1}^{k}\sum_{t=1}^{\ell} 2^{k+\ell} \lambda_{k,\ell,s,t}^{-2}\max_{\substack{1\le i<2^{k}\\ 1\le j<2^{\ell}}}\E X_{s+t,i,j}^{2}\right).
		\end{split}
	\end{equation}
	By applying the first inequality in \eqref{lambda.mn.2} and the stochastic domination assumption, we have
	\begin{equation}\label{fel14}
		\begin{split}
			0&\le a_{k,\ell}^{-1}\sum_{s=1}^{k}\sum_{t=1}^{\ell}2^{s+t}b_{2^{s+t}}\max_{1\le i< 2^{k}, 1\le j< 2^{\ell}}\mathbb{P}\left(X_{i,j}>b_{2^{s+t-2}}\right)\\
			&\le b_{2^{k+\ell}}^{-1}\sum_{s=1}^{k}\sum_{t=1}^{\ell}b_{2^{s+t}}2^{s+t}\mathbb{P}\left(|X|>b_{2^{s+t-2}}\right).
		\end{split}
	\end{equation}
	It is clear that 
	\begin{equation}\label{fel17}
		\sup_{k\ge1,\ell\ge1}	b_{2^{k+\ell}}^{-1}\sum_{s=1}^{k}\sum_{t=1}^{\ell}b_{2^{s+t}}\le C.
	\end{equation} We also have from \eqref{fel03} that
	\begin{equation}\label{fel19}
		\lim_{s\vee t\to\infty}2^{s+t}\mathbb{P}\left(|X|>b_{2^{s+t-2}}\right)=0.
	\end{equation}
	By using \eqref{fel17} and \eqref{fel19} and a double sum analogue of the Toeplitz lemma (see Lemma 2.2 in \cite{stadtmuller2011strong}), we obtain
	\begin{equation}\label{fel21}
		\begin{split}
			\lim_{k\vee \ell\to\infty}b_{2^{k+\ell}}^{-1}\sum_{s=1}^{k}\sum_{t=1}^{\ell}b_{2^{s+t}}2^{s+t}\mathbb{P}\left(|X|>b_{2^{s+t-2}}\right)=0.
		\end{split}
	\end{equation}
	Combining \eqref{fel14} and \eqref{fel21} yields
	\begin{equation}\label{fel24}
		\begin{split}
			\lim_{k\vee \ell\to\infty}a_{k,\ell}^{-1}\sum_{s=1}^{k}\sum_{t=1}^{\ell}2^{s+t}b_{2^{s+t}}\max_{1\le i< 2^{k}, 1\le j< 2^{\ell}}\mathbb{P}\left(X_{i,j}>b_{2^{s+t-2}}\right)=0.
		\end{split}
	\end{equation}
	By applying the stochastic domination assumption again, we have
	\begin{equation}\label{fel27}
		\begin{split}
			&\sum_{s=1}^{k}\sum_{t=1}^{\ell} 2^{k+\ell} \lambda_{k,\ell,s,t}^{-2}\max_{\substack{1\le i<2^{k}\\ 1\le j<2^{\ell}}}\E X_{s+t,i,j}^{2}\le \sum_{s=1}^{k}\sum_{t=1}^{\ell} 2^{k+\ell} \lambda_{k,\ell,s,t}^{-2}\left(\E X^{2}\mathbf{1}(|X|\le b_{2^{s+t}})+b_{2^{s+t}}^{2}\mathbb{P}(|X|>b_{2^{s+t}})\right)\\
			&\qquad\qquad=\sum_{s=1}^{k}\sum_{t=1}^{\ell} 2^{(k+\ell)(1-2a)}2^{-2(1/p-a)(s+t)}\left(\E X^{2}\mathbf{1}(|X|\le b_{2^{s+t}})+b_{2^{s+t}}^{2}\mathbb{P}(|X|>b_{2^{s+t}})\right)\\
			&\qquad\qquad=2^{(k+\ell)(1-2a)}\sum_{s=1}^{k}\sum_{t=1}^{\ell} 2^{(2a-1)(s+t)}\left(\frac{2^{s+t}}{b_{2^{s+t}}^{2}}\E X^{2}\mathbf{1}(|X|\le b_{2^{s+t}})+2^{s+t}\mathbb{P}(|X|>b_{2^{s+t}})\right)\\
			&\qquad\qquad:=2^{(k+\ell)(1-2a)}\sum_{s=1}^{k}\sum_{t=1}^{\ell} 2^{(2a-1)(s+t)}\left(y_{1}(s,t)+y_{2}(s,t)\right),
		\end{split}
	\end{equation}
	where
	\begin{equation*}\label{y1y2}
		y_{1}(s,t)=\frac{2^{s+t}}{b_{2^{s+t}}^{2}}\E X^{2}\mathbf{1}(|X|\le b_{2^{s+t}})\ \text{ and } \ y_{2}(s,t)=2^{s+t}\mathbb{P}(|X|>b_{2^{s+t}}).
	\end{equation*}
	By applying the Toeplitz lemma and \eqref{fel03}, we have
	\begin{equation}\label{fel31}
		\begin{split}
			y_{1}(s,t)&=\frac{1}{2^{(s+t)(2/p-1)}}\left(\sum_{j=1}^{s+t}\E X^{2}\mathbf{1}(b_{2^{j-1}}<|X|\le b_{2^{j}})+\E X^{2}\mathbf{1}(0\le |X|\le b_{1})\right)\\
			&\le \frac{1}{2^{(s+t)(2/p-1)}}\left(\sum_{j=1}^{s+t}b_{2^{j}}^{2}\mathbb{P}(|X|>b_{2^{j-1}})+b_{1}^{2}\right)\\
			&= \frac{1}{2^{(s+t)(2/p-1)}}\left(\sum_{j=1}^{s+t}2^{j(2/p-1)}2^{j}\mathbb{P}(|X|>b_{2^{j-1}})+b_{1}^{2}\right)\\
			&\to 0 \ \text{ as }\ s\vee t\to\infty.
		\end{split}
	\end{equation}
	Applying \eqref{fel03} again, we have
	\begin{equation}\label{fel32}
		\begin{split}
			y_{2}(s,t)\to 0 \ \text{ as }\ s\vee t\to\infty.
		\end{split}
	\end{equation}
	Similar to \eqref{fel21}, we conclude from \eqref{fel31}, \eqref{fel32} and the double sum analogue of the Toeplitz lemma that
	\begin{equation}\label{fel34}
		\begin{split}
			\lim_{k\vee \ell\to\infty}2^{(k+\ell)(1-2a)}\sum_{s=1}^{k}\sum_{t=1}^{\ell} 2^{(2a-1)(s+t)}\left(y_{1}(s,t)+y_{2}(s,t)\right)=0.
		\end{split}
	\end{equation}
	Combining \eqref{fel27} and \eqref{fel34} yields
	\begin{equation}\label{fel37}
		\lim_{k\vee \ell\to\infty}\sum_{s=1}^{k}\sum_{t=1}^{\ell} 2^{k+\ell} \lambda_{k,\ell,s,t}^{-2}\max_{1\le i<2^{k}, 1\le j<2^{\ell}}\E X_{s+t,i,j}^{2}=0.
	\end{equation}
	From \eqref{fel12}, \eqref{fel24} and \eqref{fel37}, we obtain \eqref{fel11}. Combining \eqref{fel07}--\eqref{fel11} yields \eqref{fel04}.
	The proof of the sufficiency part is completed.
	
	We will now prove the necessity part. Since the random variables $X_{m,n},m\ge1,n\ge1$ are symmetric, \eqref{fel04} becomes
	\begin{equation*}
		\frac{\max_{u\le m,v\le n}\left|\sum_{i=1}^{u}\sum_{j=1}^{v}X_{i,j}\right|}{(mn)^{1/p}}\overset{\mathbb{P}}{\to} 0 \text{ as }m\vee n\to\infty.
	\end{equation*}
	This implies
	\begin{equation}\label{fel39}
		\frac{\max_{i\le m, j\le n}|X_{i,j}|}{(mn)^{1/p}}\overset{\mathbb{P}}{\to} 0 \ \text{ as }\ m\vee n\to\infty.
	\end{equation}
	Applying Lemma \ref{lem.BC}, we obtain from \eqref{fel39} that
	\[\lim_{m\vee n\to\infty}mn\mathbb{P}\left(|X|>(mn)^{1/p}\right)=\lim_{m\vee n\to\infty}mn\mathbb{P}\left(|X_{1,1}|>(mn)^{1/p}\right)=0,\]
	or, equivalently, \eqref{fel03} holds.
\end{proof}

\begin{remark}\label{K3}
	In the proof of the sufficiency of Theorem \ref{thm.Feller}, we obtain from 
	\eqref{fel12}, \eqref{fel24} and \eqref{fel37} that
	\begin{equation}\label{fel41}
		K_{3}(k,\ell)\overset{\mathcal{L}_2}{\to}0\ \text{ as }\ k\vee \ell\to \infty
	\end{equation}
	which is stronger than \eqref{fel11}.
\end{remark}

Before proving Theorem \ref{thm.PR}, we state the following result which may be of independent interest.
The result involves the concept of regularly varying functions which is presented as follows. 
A real-valued function $R(\cdot )$ is said to be \textit{regularly varying} with index of regular variation
$\rho\in\R$ if it is 
a positive and measurable function on $[0,\infty)$, and for each $\lambda>0$,
\begin{equation*}\label{rv01}
	\lim_{x\to\infty}\dfrac{R(\lambda x)}{R(x)}=\lambda^\rho.
\end{equation*}
A regularly varying function with the index of regular variation $\rho=0$ is called \textit{slowly varying}.
Let $L(\cdot)$ be a slowly varying function.
Then by Theorem 1.5.13 in Bingham et al. \cite{bingham1989regular},
there exists a slowly varying function $\tilde{L}(\cdot)$, unique up to asymptotic equivalence, satisfying
\begin{equation*}
	\lim_{x\to\infty}L(x)\tilde{L}\left(xL(x)\right)=1\ \text{ and } \lim_{x\to\infty}\tilde{L}(x)L\left(x\tilde{L}(x)\right)=1.
\end{equation*}
The function $\tilde{L}$ is called the de Bruijn conjugate of $L$ (see
p. 29 in Bingham et al. \cite{bingham1989regular}). If $L(x)=\log^\gamma x$ or $L(x)=\log^\gamma(\log x)$ for some $\gamma\in\R$, then $\tilde{L}(x)=1/L(x)$.
By Proposition B.1.9 in \cite{haan2006extreme}, we can assume, without loss of generality, that
$x^\gamma L(x)$ and $x^\gamma \tilde{L}(x)$ are both strictly increasing for all $\gamma>0$.
Thereafter, for a slowly varying function $L(\cdot)$ defined on $[0,\infty)$,
we denote the de Bruijn
conjugate of $L(\cdot)$
by $\tilde{L}(\cdot)$.

\begin{proposition}\label{prop.unif}
	Let $p>0$, let $\{X_{\lambda},\lambda\in \Lambda\}$ be a family of random variables and let $L(\cdot)$ be a slowly varying function
	and $\tilde{L}(x)$ 
	the de Bruijn conjugate of $L(x)$. Then the following statements hold.
	\begin{description}
		\item[(i)] If $\{X_{\lambda},\lambda\in \Lambda\}$ is stochastically dominated by a random variables $X$ satisfying 
		\begin{equation}\label{domination01}
\E\left(|X|^pL(|X|^p)\right)<\infty,
		\end{equation}
		then $\{|X_{\lambda}|^pL(|X_{\lambda}|^p),\lambda\in \Lambda\}$
		is uniformly integrable.
		\item[(ii)] If $\{|X_{\lambda}|^pL(|X_{\lambda}|^p),\lambda\in \Lambda\}$
		is uniformly integrable, 
		then there exists a random variable $X$ with the distribution function
		\begin{equation}\label{domination02}
			F(x)=1-\sup_{\lambda \in \Lambda}\mathbb{P}(|X_{\lambda}|>x),\ x\in\R
		\end{equation}
		such that $\{X_{\lambda},\lambda\in \Lambda\}$ is stochastically dominated by $X$, and
		\begin{equation*}\label{domination04}
			\lim_{x\to\infty}x\mathbb{P}\left(|X|>x^{1/p}\tilde{L}^{1/p}(x)\right)=0.
		\end{equation*}
		\item[(iii)] If 
		\begin{equation*}\label{domination05}
			\sup_{\lambda \in \Lambda}\E\left(|X_{\lambda}|^pL(|X_{\lambda}|^p)\log_{\nu}^{(2)}|X_{\lambda}|\right)<\infty,
		\end{equation*}
		for some positive integer $\nu$, then there exists a random variable $X$ with distribution function $F(x)$ as in \eqref{domination02} 
		such that $\{X_{\lambda},\lambda\in \Lambda\}$ is stochastically dominated by $X$, and \eqref{domination01} holds.
	\end{description}
\end{proposition}
\begin{proof}
	The proof of Proposition \ref{prop.unif} is similar to that of Theorem 3.1 in \cite{thanh2023new}. We omit the details.
\end{proof}

We will now present the proof of Theorem \ref{thm.PR}.

\begin{proof}[Proof of Theorem \ref{thm.PR}]
	We first prove the sufficiency part. As before, we can assume that $X_{m,n},m\ge1,n\ge1$ are nonnegative.
	Set
	\[b_{n}=n^{1/p},\ X_{z,s,m,n}=X_{m,n}\mathbf{1}(X_{m,n}\le z^{1/p} b_{2^{s}})+z^{1/p} b_{2^{s}}\mathbf{1}(X_{m,n}> z^{1/p} b_{2^{s}}),\ z>0,s\ge0,m\ge1,n\ge1.\]
	For $m\ge 1, n\ge1$, we have
	\begin{equation}\label{mean.32}
		\begin{split}
			\E \left(\dfrac{1}{(mn)^{1/p}}\max_{\substack{u\le m\\ v \le n}}\left|\sum_{i=1}^{u}\sum_{j=1}^{v}(X_{i,j}-\E X_{i,j})\right|\right)^p&=\int_{0}^{\infty}\mathbb{P}\left(\dfrac{1}{(mn)^{1/p}}\max_{\substack{u\le m\\ v \le n}}\left|\sum_{i=1}^{u}\sum_{j=1}^{v}(X_{i,j}-\E X_{i,j})\right|>z^{1/p}\right)\dx z\\
			&=\int_{0}^{1}\mathbb{P}\left(\dfrac{1}{(mn)^{1/p}}\max_{\substack{u\le m\\ v \le n}}\left|\sum_{i=1}^{u}\sum_{j=1}^{v}(X_{i,j}-\E X_{i,j})\right|>z^{1/p}\right)\dx z\\
			&\quad+\int_{1}^{\infty}\mathbb{P}\left(\dfrac{1}{(mn)^{1/p}}\max_{\substack{u\le m\\ v \le n}}\left|\sum_{i=1}^{u}\sum_{j=1}^{v}(X_{i,j}-\E X_{i,j})\right|>z^{1/p}\right)\dx z\\
			&:=R_1(m,n)+R_2(m,n).
		\end{split}
	\end{equation}
	By Proposition \ref{prop.unif} (i) and (ii) with $L(x)\equiv 1$, there exists a random variable $X$ such that
	the array $\{X_{m,n},m\ge1,n\ge1\}$ is stochastically dominated by $X$ and
	\eqref{fel03} holds. Applying Theorem \ref{thm.Feller}, we obtain
	the WLLN
	\begin{equation}\label{mean.35a}
		\frac{1}{(mn)^{1/p}}\max_{u\le m, v\le n}\left|\sum_{i=1}^{u}\sum_{j=1}^{v}\left(X_{i,j}-\E (X_{i,j}\mathbf{1}(X_{i,j}\le (mn)^{1/p}))\right)\right|\overset{\mathbb{P}}{\to}0\ \text{ as }\ m\vee n\to\infty.
	\end{equation}
	By applying the stochastic domination, \eqref{pr03}
and the Lebesgue dominated convergence theorem, we have
	\begin{equation}\label{mean.35c}
	\frac{1}{(mn)^{1/p}}\max_{u\le m, v\le n}\left|\sum_{i=1}^{u}\sum_{j=1}^{v}\E (X_{i,j}\mathbf{1}(X_{i,j}>(mn)^{1/p}))\right|\le \E\left(|X|^p\mathbf{1}(X_{i,j}>(mn)^{1/p})\right)\to0\ \text{ as }\ m\vee n\to\infty.
\end{equation}
Combining \eqref{mean.35a} and \eqref{mean.35c} yields 
	\begin{equation}\label{mean.35}
	\frac{1}{(mn)^{1/p}}\max_{u\le m, v\le n}\left|\sum_{i=1}^{u}\sum_{j=1}^{v}\left(X_{i,j}-\E X_{i,j}\right)\right|\overset{\mathbb{P}}{\to}0\ \text{ as }\ m\vee n\to\infty.
\end{equation}
	By \eqref{mean.35} and the Lebesgue dominated convergence theorem, we have $\lim_{m\vee n\to\infty}R_{1}(m,n)=0.$ Therefore, in view of \eqref{mean.32}, it remains to prove
	that $\lim_{m\vee n\to\infty}R_{2}(m,n)=0.$ 		
	For $n\ge 1$, $m\ge1$, let $k,\ell$ be integer numbers such that $2^{k-1}\le m<2^{k}$ and $2^{\ell-1}\le m<2^{\ell}$.
	Then
	\begin{equation}\label{mean.41}
		\begin{split}
			R_2(m,n)&= \int_{1}^{\infty}\mathbb{P}\left(\max_{u\le m, v\le n}\left|\sum_{i=1}^{u}\sum_{j=1}^{v}(X_{i,j}-\E X_{i,j})\right|>z^{1/p}(mn)^{1/p}\right)\dx z\\
			&\le \int_{1}^{\infty}\mathbb{P}\left(\max_{u< 2^{k}, v< 2^{\ell}}\left|\sum_{i=1}^{u}\sum_{j=1}^{v}(X_{i,j}-\E X_{i,j})\right|>z^{1/p}b_{2^{k+\ell}}/4\right)\dx z\\
			&\le R_{2,1}(k,\ell)+R_{2,2}(k,\ell)+R_{2,3}(k,\ell),
		\end{split}
	\end{equation}
	where
	\[R_{2,1}(k,\ell)=\int_{1}^{\infty}\mathbb{P}\left(\max_{u< 2^{k}, v< 2^{\ell}}\left|\sum_{i=1}^{u}\sum_{j=1}^{v}(X_{i,j}- X_{z,k+\ell,i,j})\right|>z^{1/p}b_{2^{k+\ell}}/12\right)\dx z,\]
	\[R_{2,2}(k,\ell)=\int_{1}^{\infty}\mathbb{P}\left(\max_{u< 2^{k}, v< 2^{\ell}}\left|\sum_{i=1}^{u}\sum_{j=1}^{v}\E(X_{i,j}- X_{z,k+\ell,i,j})\right|>z^{1/p}b_{2^{k+\ell}}/12\right)\dx z,\]
	\[R_{2,3}(k,\ell)=\int_{1}^{\infty}\mathbb{P}\left(\max_{u< 2^{k}, v< 2^{\ell}}\left|\sum_{i=1}^{u}\sum_{j=1}^{v}(X_{z,k+\ell,i,j}- \E X_{z,k+\ell,i,j})\right|>z^{1/p}b_{2^{k+\ell}}/12\right)\dx z.\]
	By applying the stochastic domination, \eqref{pr03}
	and the Lebesgue dominated convergence theorem, we have
	\begin{equation*}
		\begin{split}
			\int_{1}^{\infty}\mathbb{P}\left(\max_{\substack{u< 2^{k}\\ v< 2^{\ell}}}\left|\sum_{i=1}^{u}\sum_{j=1}^{v}(X_{i,j}-X_{z,k+\ell,i,j})\right|>z^{1/p}b_{2^{k+\ell}}/12\right)\dx z
			&\le \int_{1}^{\infty}\sum_{i=1}^{2^{k}}\sum_{j=1}^{2^{\ell}} \mathbb{P}\left(X_{i,j}>z^{1/p}b_{2^{k+\ell}}\right)\dx z\\
			&\le \frac{1}{b_{2^{k+\ell}}^p}\sum_{i=1}^{2^{k}}\sum_{j=1}^{2^{\ell}} \E\left(X_{i,j}^p\mathbf{1}\left(X_{i,j}>b_{2^{k+\ell}}\right)\right)\\
			&\le \E\left(|X|^p\mathbf{1}\left(|X|>b_{2^{k+\ell}}\right)\right)\\
			&\to 0 \text{ as }k\vee \ell\to\infty,
		\end{split}
	\end{equation*}
	and
	\begin{equation*}
		\begin{split}
			\sup_{z\ge 1}\frac{1}{z^{1/p}b_{2^{k+\ell}}}\max_{u< 2^{k},v< 2^{\ell}}\left|\sum_{i=1}^{u}\sum_{j=1}^{v}\E(X_{i,j}-X_{z,k+\ell,i,j})\right|
			&\le 
			\sup_{z\ge 1}\frac{1}{z^{1/p}b_{2^{k+\ell}}}\sum_{i=1}^{2^{k}}\sum_{j=1}^{2^{\ell}}\E|X_{i,j}-X_{z,k+\ell,i,j}|\\
			&\le\frac{1}{b_{2^{k+\ell}}}\sum_{i=1}^{2^{k}}\sum_{j=1}^{2^{\ell}}\E\left(X_{i,j}\mathbf{1}\left(X_{i,j}> b_{2^{k+\ell}}\right)\right)\\
			&\le \frac{1}{b_{2^{k+\ell}}^{p}}\sum_{i=1}^{2^{k}}\sum_{j=1}^{2^{\ell}}\E\left(X_{i,j}^p\mathbf{1}(X_{i,j}>b_{2^{k+\ell}})\right)\\
			&\le \E\left(|X|^p\mathbf{1}(|X|>b_{2^{k+\ell}})\right)\\
			&\to 0 \text{ as }k\vee \ell\to\infty
		\end{split}
	\end{equation*}
	which, respectively, yields $\lim_{k\vee \ell\to\infty}R_{2,1}(k,\ell)=0$ and $R_{2,2}(k,\ell)=0$ for all large $k\vee \ell$. 
	Finally, by using Tonelli’s theorem and proceeding in a similar manner as the argument in
	the proof of \eqref{fel41},
	we obtain $\lim_{k\vee \ell\to\infty}R_{2,3}(k,\ell)=0$.
	Therefore, \eqref{mean.41} ensures that $\lim_{m\vee n\to\infty}R_{2}(m,n)=0$
	which ends the proof of the sufficiency part of the theorem.
	
	We will now prove the necessity part. Assume that \eqref{pr05} holds. Then
	\[\begin{split}
	\frac{\E|X-\mu|^p}{n}&=\frac{\E\left|\sum_{i=1}^{1}\sum_{j=1}^{n}(X_{i,j}-\mu)-\sum_{i=1}^{1}\sum_{j=1}^{n-1}(X_{i,j}-\mu)\right|^p}{n}\\
	&\le 2^{p-1}\left(\frac{\E\left|\sum_{i=1}^{1}\sum_{j=1}^{n}(X_{i,j}-\mu)\right|^p}{n}+\frac{\E\left|\sum_{i=1}^{1}\sum_{j=1}^{n-1}(X_{i,j}-\mu)\right|^p}{n} \right)\to 0 \text{ as } n\to\infty,
	\end{split}\]
	and therefore $\E|X-\mu|^p<\infty$, which, in turn, implies that \eqref{pr03} holds. Applying the sufficiency part, we obtain 
	\begin{equation}\label{pr07}
		\E\left|\frac{\sum_{i=1}^{m}\sum_{j=1}^{n}X_{i,j}}{(mn)^{1/p}}-(mn)^{1-1/p}\E X\right|^p \to 0 \text{ as }m\vee n\to\infty.
	\end{equation}
	By using \eqref{pr05} and \eqref{pr07}, we have
	\begin{equation}\label{pr09}
		\begin{split}
			\left|(mn)^{1-1/p}(\E X-\mu)\right|^p&\le 2^{p-1}\E\left|\frac{\sum_{i=1}^{m}\sum_{j=1}^{n}X_{i,j}}{(mn)^{1/p}}-(mn)^{1-1/p}\E X\right|^p\\
			&\qquad+ 2^{p-1}\E\left|\frac{\sum_{i=1}^{m}\sum_{j=1}^{n}X_{i,j}}{(mn)^{1/p}}-(mn)^{1-1/p}\mu \right|^p \to 0 \text{ as }m\vee n\to\infty.
		\end{split}
	\end{equation}
	Since $1-1/p\ge 0$, \eqref{pr09} ensures that $\E X=\mu$. The proof of the necessity part is completed.	
\end{proof}

\section{Some corollaries and further remarks}\label{sec:remarks}

\subsection{Limit theorems under bounded moment conditions}
From Proposition \ref{prop.unif}, it follows that certain bounded moment conditions on the random field can accomplish 
the stochastic domination condition.
This illustrates the flexibility of the stochastic domination condition in comparison to the identical distribution condition. 
Specifically, by using Proposition \ref{prop.unif} and the results in Section \ref{sec.intro} (Theorems \ref{thm.BKHRE}, \ref{thm.Feller} and \ref{thm.PR}), 
we obtain the following corollaries. Details of the proof will be omitted. 

\begin{corollary}\label{thm.main.BKHRE.unif}
	Let $p\ge 1$, $\alpha>1/2$ and let 
	$\{X_{m,n}, \ m\ge 1,\ n \geq 1\}$ be a double array of random variables.
	Assume that the $\{X_{m,n}, \ m\ge 1,\ n \geq 1\}$ satisfies Condition $(H_{2q})$ with $q=1$ if $1\le p<2$ and $q>(\alpha p-1)/(2\alpha-1)$ if $p\ge 2$.
	If 	\[\sup_{m\ge1,n\ge1}\E\left(|X_{m,n}|^p \log|X_{m,n}|\log_{\nu}^{(2)}|X_{m,n}|\right)<\infty\]
	for some positive integer $\nu$, then \eqref{mz05} holds.
\end{corollary}

\begin{corollary}\label{thm.main.PR.unif}
	
	Let $1\le p<2$ and let 
	$\{X_{m,n}, \ m\ge 1,\ n \geq 1\}$ be a double array of random variables satisfying Condition $(H_{2})$.	
	If 	\[\sup_{m\ge1,n\ge1}\E\left(|X_{m,n}|^p \log_{\nu}^{(2)}|X_{m,n}|\right)<\infty\]
	for some positive integer $\nu$, then \eqref{pr04} holds.
\end{corollary}

\begin{corollary}\label{thm.main.Feller.unif}
	Let $1\le p<2$ and let 
	$\{X_{m,n}, \ m\ge 1,\ n \geq 1\}$ be a double array of random variables satisfying Condition $(H_{2})$.	
	If 	$\{|X_{m,n}|^{p}, \ m\ge 1,\ n \geq 1\}$ is uniformly integrable, then
	\begin{equation}\label{fel04.unif}
		\frac{\max_{u\le m,v\le n}\left|\sum_{i=1}^{u}\sum_{j=1}^{v}(X_{i,j}-\E X_{i,j})\right|}{(mn)^{1/p}}\overset{\mathbb{P}}{\to} 0 \text{ as }m\vee n\to\infty.
	\end{equation}
\end{corollary}

\begin{problem}\label{open3}
	In Corollary \ref{thm.main.Feller.unif}, if the random variables $X_{m,n}, \ m\ge 1,\ n \geq 1$ are
	independent, then by the method in Pyke and Root \cite{pyke1968convergence}, we can obtain
	convergence in mean of order $p$ in \eqref{fel04.unif}.
	If we only assume $X_{m,n}, \ m\ge 1,\ n \geq 1$ satisfy Condition $(H_2)$, then we do not know whether or not
	the convergence in mean of order $p$ prevails in \eqref{fel04.unif}. To the best of our knowledge, this problem is
	unsolved even in the case of dimension one.
\end{problem}

\subsection{Limit theorems for dependent random fields with regularly varying norming constants}

Theorems \ref{thm.BKHRE}, \ref{thm.Feller} and \ref{thm.PR} can be extended to the case where the norming constants are regularly varying. 
For instance, we have an extension of Theorem \ref{thm.BKHRE} as follows. The proof employs some properties of slowly varying functions presented in \cite{anh2021marcinkiewicz,thanh2020theBaum,thanh2023weak,thanh2023hsu} as well as the technique developed in Sections \ref{sec.maximal.inequalities} and \ref{sec.bk}. 
We leave the details to the interested reader.

\begin{theorem}\label{thm.BKHRE.regular}
	Let $p\ge 1$, $1/2<\alpha\le 1$, $\alpha p\ge 1$ and let 
	$\{X_{m,n}, \ m\ge 1,\ n \geq 1\}$ be a double array of identically distributed random variables. 
	Let $L(x)\ge 1$ be an increasing slowly varying function and $\tilde{L}(x)$ 
	the de Bruijn conjugate of $L(x)$.
	Assume that the array $\{X_{m,n}, \ m\ge 1,\ n \geq 1\}$ satisfies Condition $(H_{2q})$ with $q=1$ if $1\le p<2$ and $q>(\alpha p-1)/(2\alpha-1)$ if $p\ge 2$.
	Then
	\begin{equation*}\label{bk.regular}
		\sum_{m=1}^{\infty}\sum_{n=1}^{\infty}
		(mn)^{\alpha p-2}\mathbb{P}\left(\max_{u\le m,v\le n}\left|\sum_{i=1}^{u}\sum_{j=1}^{v}X_{i,j}\right|>\varepsilon (mn)^{\alpha}\tilde{L}((mn)^{\alpha})\right)<\infty \text{ for all } \varepsilon>0
	\end{equation*}
	if and only if
	\begin{equation*}
		\E X_{1,1}=0 \text{ and }\	\E\left(|X_{1,1}|^p L^p(|X_{1,1}|) \log|X_{1,1}|\right)<\infty				.
	\end{equation*}
\end{theorem}

\subsection{Further remarks on limit theorems for mixing random fields and negatively dependent random fields}\label{subsec.mixing}

In this subsection,	for any two $\sigma$-fields $\mathcal{A},\mathcal{B}\subset \mathcal{F}$, we define
the maximal coefficient of correlation
\[\rho(\mathcal{A},\mathcal{B})=\sup\frac{|\cov (XY)|}{\left(\var(X)\var(Y)\right)^{1/2}},\]
where the $\sup$ is taken over all pairs of
random variables $X\in\mathcal{L}_2(\mathcal{A})$ and $Y\in\mathcal{L}_2(\mathcal{B})$, and $0/0$ is interpreted to be $0$.

The concepts of $\rho^*$-mixing and $\rho'$-mixing random fields were introduced by Bradley and Utev \cite{bradley1994second} (see also in Bradley \cite{bradley2010dependence}, Bradley and Tone \cite{bradley2017central}).
Let $\mathbb{Z}_{+}$ be the set of positive integers and let $d\in\mathbb{Z}_{+}$.
Let $\mathbb{Z}^{d}_{+}$ denote the positive integer $d$-dimensional lattice points.
The notation $\mathbf {m \prec n}$ (or $\mathbf{n}\succ\mathbf{m}$), 
where $\mathbf m$ = $(m_1, m_2,..., m_d)$ and $\mathbf n$ = $(n_1, n_2,..., n_d) \in \mathbb{Z}^{d}_{+},$ 
means that $m_i \le n_i, 1 \le i \le d$. 
For $\mathbf{n}$ = $(n_1, n_2,..., n_d) \in \mathbb{Z}^{d}_{+},$ let $\|\mathbf{n}\|=(n_{1}^{2}+\cdot+n_{d}^{2})^{1/2}$
denote the Euclidean norm. Let $\mathcal{X}=\left\{ X_{\mathbf{n}}, \mathbf{n} \in \mathbb{Z}^{d}_{+} \right\}$ be a $d$-dimensional random field.
For two nonempty disjoint subsets $S_1$ and $S_2$ of $\mathbb{Z}^{d}_{+}$, denote
\[{\rm{dist}}(S_1,S_2):=\inf_{\mathbf{n}\in S_1,\mathbf{m}\in S_2}\|\mathbf{n}-\mathbf{m}\|,\]
and
\[\rho(S_1,S_2):=\rho(\sigma(X_{\mathbf{n}},\mathbf{n}\in S_1),\sigma(X_{\mathbf{n}},\mathbf{n}\in S_2)).\]
For $n\ge 1$, we define
\begin{equation*}\label{rho1}
	\rho^*(\mathcal{X},n)=\sup \{\rho(S_1,S_2): {\rm{dist}}(S_1,S_2)\ge n\}
\end{equation*}
and
\begin{equation}\label{rho2}
	\rho'(\mathcal{X},n)=\sup \rho(S_1,S_2),
\end{equation}
where in \eqref{rho2}, the $\sup$ is taken over all pairs of nonempty disjoint subsets $S_1$ and $S_2$ of $\mathbb{Z}^{d}_{+}$ of the form
\[S_1=\{\mathbf{n}=(n_1, n_2,..., n_d) \in \mathbb{Z}^{d}_{+}:n_i\in Q_1\}\]
and
\[S_2=\{\mathbf{n}=(n_1, n_2,..., n_d) \in \mathbb{Z}^{d}_{+}:n_i\in Q_2\},\]
where $i=1,\ldots d$, and $Q_1$ and $Q_2$ are two nonempty disjoint subsets of $\mathbb{Z}^{1}_{+}$ satisfying ${\rm{dist}}(Q_1,Q_2)\ge n$.
As noted by Bradley and Utev \cite{bradley1994second}, $\rho^*$ is based on ``general'' disjoint sets $S_1$ and $S_2$ whereas $\rho'$ 
is based on disjoint ``one-dimensional cylinder sets'' $S_1$ and $S_2$. 
It is clear that 
$0\le \rho'(n)\le \rho^*(n)\le 1$ for all $n\ge 1$. The random field $\mathcal{X}$ is said to be 
\textit{$\rho^*$-mixing} (resp., \textit{$\rho'$-mixing}) if $\lim_{n\to\infty}\rho^{*}(\mathcal{X},n)=0$ (resp., $\lim_{n\to\infty}\rho'(\mathcal{X},n)=0$).
If $\lim_{n\to\infty}\rho^{*}(\mathcal{X},n)<1$, then the array $\mathcal{X}$ satisfies Condition $H_{2q}$ for all $q\ge1$ (see Theorem 4 of Peligrad and Gut \cite{peligrad1999almost}).
If $\lim_{n\to\infty}\rho'(\mathcal{X},n)<1$, then the array $\mathcal{X}$ satisfies Condition $H_{2q}$ for all $q\ge1$ (see Theorem
29.30 of Bradley \cite{bradley2007introduction}).

Limit theorems for mixing random fields were studied extensively
by various authors. We refer to Bradley \cite{bradley1992spectral,bradley1993equivalent}, Bradley and Tone \cite{bradley2017central}
for the central limit theorems for $\rho^{*}$-mixing and $\rho'$-mixing random fields,
Kuczmaszewska and Lagodowski \cite{kuczmaszewska2011convergence}, Peligrad and Gut \cite{peligrad1999almost}
and the references therein for the Hsu--Robbins--Erd\"{o}s--Splitzer--Baum--Katz-type 
theorem and SLLNs for $\rho^{*}$-mixing random fields.
However, to the best of our knowledge, there
are no results in the literature on complete convergence or SLLNs for
$\rho'$-mixing random fields. Let $\mathcal{X}=\left\{ X_{\mathbf{n}}, \mathbf{n} \in \mathbb{Z}^{d}_{+} \right\}$ be a $d$-dimensional random field.
A Rosenthal-type maximal inequality for the random field $\mathcal{X}$ under the condition 
$\lim_{n\to\infty}\rho^{*}(\mathcal{X},n)<1$ was provided by Peligrad and Gut \cite{peligrad1999almost}
but such an inequality is not available for the $\rho'$-mixing case.
This prevents us from using existing methods to establish laws of large numbers for the maximum of multiple sums for $\rho'$-mixing random fields.

As mentioned in the above, if $\lim_{n\to\infty}\rho'(\mathcal{X},n)<1$, then $\mathcal{X}$ satisfies Condition $H_{2q}$ 
for all $q\ge1$. Therefore, all results in Sections 
\ref{sec.intro}--\ref{sec.maximal.inequalities} hold true for dependent
random fields satisfying $\lim_{n\to\infty}\rho'(\mathcal{X},n)<1$.
For the case where $d=1$, we have $\rho'(n)= \rho^*(n)$ for all $n\ge1$ and thus there is no difference between
$\rho^{*}$-mixing sequences and $\rho'$-mixing sequences.
However, for the case where $d\ge 2$, It was shown by Bradley \cite[Theorem 1.9]{bradley2010dependence} that
for all nonincreasing sequence
$\{c_n,n\ge 1\}\subset [0,1]$, there exists a strictly stationary random field
$\left\{ X_{\mathbf{n}}, \mathbf{n} \in \mathbb{Z}^{d}_{+} \right\}$ such that $\rho^*(n)=1$ for all $n\ge1$
and $\rho'(n)=c_n$ for all $n\ge2$.  Therefore for the case of dimension $d\ge2$, our
result on the Baum--Katz--Erd\"{o}s--Hsu--Robbins-type theorem under condition $\lim_{n\to\infty}\rho'(\mathcal{X},n)<1$
significantly improves the Peligrad and Gut \cite{peligrad1999almost} result in the sense that it 
cannot be derived from the Peligrad and Gut \cite{peligrad1999almost} result
for dependent random fields with condition $\lim_{n\to\infty}\rho^{*}(\mathcal{X},n)<1$.

Kuczmaszewska and Lagodowski \cite{kuczmaszewska2011convergence} used the method in the Peligrad and Gut \cite{peligrad1999almost} 
to establish the Hsu--Robbins--Erd\"{o}s--Spitzer--Baum--Katz-type theorem for
negatively associated random fields. It is well known that
negative association is strictly stronger than
pairwise negative dependence (see, \cite[Property P3 and Remark 2.5]{joag1983negative}).
The Rosenthal-maximal inequalities also hold for negatively associated mean zero random variables (see, e.g., Shao \cite{shao2000comparison}, Giap et al. \cite{giap2022some})
but for pairwise negatively dependent mean zero random variables, \eqref{kd.inequality} is not valid even in the case of
dimension one. 
By Lemma 1 (ii) and Lemma 3 of Lehmann \cite{lehmann1966some},
pairwise negatively dependent random variables satisfy 
Condition $(H_2)$. Therefore, Theorem \ref{thm.BKHRE} for the case $1\le p<2$, and Theorems \ref{thm.Feller} and \ref{thm.PR}
can be applied to the pairwise negatively dependent random fields. As stated in Section \ref{sec.intro}, these results are new even when the underlying random
variables are pairwise independent.

There is another dependence structure called extended negative dependence (see, e.g., Chen et al. \cite{chen2010strong}),
which is strictly weaker than negative association. 
Lemmas 2.1 and 2.3 of Shen et al. \cite{shen2017weak} ensure that 
extended negative dependence possesses Condition $(H_{2q})$ for all $q\ge 1$.
Therefore, our result in Sections \ref{sec.intro}--\ref{sec.maximal.inequalities}
can also be applied to this dependence structure. We note that a Kolmogorov--Doob-type maximal inequality or
a Rosenthal-type maximal inequality is not available for extended negatively dependent random variables and negatively dependent random variables, even
in the case of dimension one. Theorems \ref{thm.BKHRE}, \ref{thm.Feller} and \ref{thm.PR} for these two dependence structures 
have never appeared
in the literature. 
Chen et al. \cite{chen2010strong} were apparently the first to 
establish the Kolmogorov SLLN for extended negatively dependent random variables in the case of dimension one.

Finally, we remark that even for the $\rho^{*}$-mixing case with condition $\lim_{n\to\infty}\rho^{*}(\mathcal{X},n)<1$, the
Rosenthal maximal inequality
provided by Peligrad and Gut \cite{peligrad1999almost} is not sharp since the bound of
the second moment of the maximum $d$-index sums has an additional factor $(\log|\mathbf{n}|)^{2d}$ (see Corollary 2 in
Peligrad and Gut \cite{peligrad1999almost}). Therefore, the Peligrad and Gut \cite{peligrad1999almost} result
on the Baum--Katz--Erd\"{o}s--Hsu--Robbins-type 
theorem has to require $\alpha>1/p$, and so we cannot
derive the Marcinkiewicz--Zygmund SLLN for random fields from their result.
In Peligrad and Gut \cite{peligrad1999almost}, the authors only obtained the Kolmogorov SLLN
(i.e., the case $p=1$ in the Marcinkiewicz--Zygmund SLLN) by using the Etemadi subsequences method (see \cite[Theorem 6]{peligrad1999almost}).
Similar to Peligrad and Gut \cite{peligrad1999almost}, Kuczmaszewska and Lagodowski \cite{kuczmaszewska2011convergence} also required $\alpha>1/p$ in
their result (see \cite[Theorem 3.2]{kuczmaszewska2011convergence}).\\

\begin{center}
	\textbf{Acknowledgments} 
\end{center}
The author is grateful to the Reviewer for carefully reading the manuscript and for offering
valuable comments and suggestions which enabled him to greatly improve the presentation of the paper.
In particular, the Reviewer kindly brought to the author’s attention the $2s$-tuplewise independent random variables leading to Remark \ref{rem.open}.
The author also grateful to Vu Thi Ngoc Anh and Nguyen Chi Dzung for useful comments and remarks. \\

\noindent\textbf{Declaration of competing interest:}
The author declares that he has no known competing financial interests or personal relationships that could have
appeared to influence the work reported in this paper.\\

\noindent \textbf{Funding:}
The author did not receive support from any organization for this work.

\appendix

\section{}
\renewcommand{\thesection}{\Alph{section}}

In this section, we will present two technical lemmas and prove Claim \ref{claim.estimate}.

\begin{lemma}\label{lem.BC}
	Let $\{X_{m,n},m\ge1,n\ge1\}$ be a double array of identically distributed
	random variables satisfying Condition $(H_{2})$ and let $\{b_{m,n},m\ge1,n\ge1\}$ be a double array of
	positive constants.
	If \begin{equation}\label{bc02}
		\frac{\max_{1\le i\le m, 1\le j\le n}|X_{i,j}|}{b_{m,n}}\overset{\mathbb{P}}{\to} 0 \ \text{ as }\ m\vee n\to\infty,
	\end{equation}
	then for all $\varepsilon>0$, there exists $n_0$ such that
	\begin{equation}\label{bc03}
		mn\mathbb{P}(|X_{1,1}|>b_{m,n}\varepsilon)\le C \mathbb{P}\left(\max_{1\le i\le m, 1\le j\le n}|X_{i,j}|>b_{m,n}\varepsilon\right)\ \text{ for all }\ m\vee n\ge n_0,
	\end{equation}
	and so
	\begin{equation*}\label{bc04}
		mn\mathbb{P}(|X_{1,1}|>b_{m,n}\varepsilon)\to 0 \ \text{ as }\ m\vee n\to\infty.
	\end{equation*}
\end{lemma}

\begin{proof}
	Let $\varepsilon>0$ be arbitrary. It follows from \eqref{bc02} that
	\begin{equation}\label{bc05}
		\lim_{m\vee n\to\infty}\mathbb{P}\left(\max_{k\le m,\ell\le n} X_{k,\ell}^{+}>b_{m,n}\varepsilon\right)=\lim_{m\vee n\to\infty}\mathbb{P}\left(\bigcup_{k=1}^m\bigcup_{\ell=1}^n \left(X_{k,\ell}^{+}>b_{m,n}\varepsilon\right)\right)=0.
	\end{equation}
	Since the array $\{X_{m,n}, \ m\ge 1,\ n \geq 1\}$ is comprised of identically distributed random variables and satisfies Condition $(H_{2})$, we can apply Proposition 2.5 in \cite{thanh2023hsu} for events $\{(X_{k,\ell}^{+}>b_{m,n}\varepsilon),1\le k\le m,1\le \ell\le n\}$
	to obtain
	\begin{equation}\label{bc07}
		\left(1-\mathbb{P}\left(\max_{k\le m,\ell\le n} X_{k,\ell}^{+}>b_{m,n}\varepsilon\right)\right)^2\sum_{k= 1}^m\sum_{\ell=1}^n \mathbb{P}(X_{k,\ell}^{+}>b_{m,n}\varepsilon)\le C\mathbb{P}\left(\max_{k\le m,\ell\le n} X_{k,\ell}^{+}>b_{m,n}\varepsilon\right).
	\end{equation}
	It follows from \eqref{bc05} and \eqref{bc07} that there exists a positive integer $n_1$ such that
	\begin{equation}\label{bc08}
		mn\mathbb{P}(X_{1,1}^{+}>b_{m,n}\varepsilon)=\sum_{k= 1}^m\sum_{\ell=1}^n \mathbb{P}(X_{k,\ell}^{+}>b_{m,n}\varepsilon)\le C\mathbb{P}\left(\max_{k\le m,\ell\le n} X_{k,\ell}^{+}>b_{m,n}\varepsilon\right)
	\end{equation}
	whenever $m\vee n\ge n_1$. By using the same arguments, we also have
	\begin{equation}\label{bc09}
		mn\mathbb{P}(X_{1,1}^{-}>b_{m,n}\varepsilon)\le C\mathbb{P}\left(\max_{k\le m,\ell\le n} X_{k,\ell}^{-}>b_{m,n}\varepsilon\right)
	\end{equation}
	whenever $m\vee n\ge n_2$ for some positive integer $n_2$.	Letting $n_0=\max\{n_1,n_2\}$ and combining \eqref{bc08} and \eqref{bc09}, we obtain \eqref{bc03}.
\end{proof}

\begin{lemma}\label{lem.moment.estimate03}
	Let $\alpha>0,\ q>0,\ 0<p<q$ and let $X$ be a random variable.
	Then the following statements are equivalent:
	\begin{description}
		\item[(i)] $\E\left(|X|^p\log|X|\right)<\infty.$
		\item[(ii)] $\sum_{m=1}^{\infty}\sum_{n=1}^{\infty}(mn)^{\alpha p-1}\mathbb{P}\left(|X|>(mn)^{\alpha}\right)<\infty.$
		\item[(iii)] $\sum_{m=1}^{\infty}\sum_{n=1}^{\infty}2^{(m+n)\alpha p}\mathbb{P}\left(|X|>2^{(m+n)\alpha}\right)<\infty.$
		\item[(iv)] $\sum_{m=1}^{\infty}\sum_{n=1}^{\infty}(mn)^{\alpha (p-q)-1} \E\left(|X|^q\mathbf{1}\left(|X|\le (mn)^{\alpha}\right)\right)<\infty.$
		\item[(v)] $\sum_{m=1}^{\infty}\sum_{n=1}^{\infty}2^{(m+n)\alpha (p-q)} \E\left(|X|^q\mathbf{1}\left(|X|\le 2^{(m+n)\alpha}\right)\right)<\infty.$
	\end{description}
\end{lemma}
\begin{proof}
	The equivalence of (i) and (ii) is a special case of Lemma 2.1 in Gut \cite{gut1978marcinkiewicz}.
	The proof of the equivalence of (i) and (iv) is similar.
	The equivalence of (ii) and (iii), and the equivalence of (iv) and (v) are obvious.
\end{proof}

Finally, we present the proof of Claim \ref{claim.estimate} which was used in the proof of Theorem \ref{thm.maximal.ineq1}.

\begin{proof}[Proof of Claim \ref{claim.estimate}]
	
	For $m\ge 1,\ n\ge 1$, $1\le u<2^m$, $1\le v<2^n$, $0\le s\le m$, $0\le t\le n$, set 
	\[k_{u,s}=\lfloor u/2^s \rfloor,\ \ell_{v,t}=\lfloor v/2^t \rfloor,\ u_s=k_{u,s}2^s, \ v_t=\ell_{v,t}2^t,\]
	\begin{equation*}
		\begin{split}
		T_{s-1,t,u_{s-1},v}&=S_{s-1,t,u_{s-1},v}-S_{s-1,t,u_{s},v}\ (s\ge 1),
		\end{split}
	\end{equation*}
	and
	\begin{equation*}
		\begin{split}
T_{s,t,u,v}^{*}&=S_{s,t,u,v}-S_{s-1,t,u,v}-S_{s,t,u_{s},v}+S_{s-1,t,u_{s},v}\ (s\ge 1).
		\end{split}
	\end{equation*}
	Then $u_0=u$, $v_0=v$ and $u_m=v_n=0$.
	For all $m\ge 1,n\ge 1$, $1\le u<2^m,1\le v<2^n$, we have
	\begin{equation}\label{mz31a}
		\begin{split}
			S_{m,n,u,v}&=\sum_{s=1}^{m} \left(S_{s-1,n,u_{s-1},v}-S_{s-1,n,u_{s},v} \right)\\
			&\quad+\sum_{s=1}^{m}\left(S_{s,n,u,v}-S_{s-1,n,u,v}-S_{s,n,u_{s},v}+S_{s-1,n,u_{s},v} \right)\\
			&=\sum_{s=1}^{m} T_{s-1,n,u_{s-1},v}+\sum_{s=1}^{m} T_{s,n,u,v}^{*}.
		\end{split}
	\end{equation}
	Applying the above decomposition again for the second and the fourth indices, we have
	\begin{equation}\label{mz31b}
		\begin{split}
			T_{s-1,n,u_{s-1},v}&=\sum_{t=1}^n\left(T_{s-1,t-1,u_{s-1},v_{t-1}}-T_{s-1,t-1,u_{s-1},v_{t}}\right)\\
			&\quad+\sum_{t=1}^n\left(T_{s-1,t,u_{s-1},v}-T_{s-1,t-1,u_{s-1},v}- T_{s-1,t,u_{s-1},v_{t}}+T_{s-1,t-1,u_{s-1},v_{t}}\right),
		\end{split}
	\end{equation}
	and
	\begin{equation}\label{mz31c}
		\begin{split}
			T_{s,n,u,v}^{*}&=\sum_{t=1}^n\left(T_{s,t-1,u,v_{t-1}}^{*}-T_{s,t-1,u,v_{t}}^{*}\right)\\
			&\quad+\sum_{t=1}^n\left(T_{s,t,u,v}^{*}-T_{s,t-1,u,v}^{*}- T_{s,t,u,v_{t}}^{*}+T_{s,t-1,u,v_{t}}^{*}\right).
		\end{split}
	\end{equation}
	Combining \eqref{mz31a}--\eqref{mz31c} yields
	\begin{equation}\label{mz31d}
		\begin{split}
			S_{m,n,u,v}&=\sum_{s=1}^{m}\sum_{t=1}^n\left(T_{s-1,t-1,u_{s-1},v_{t-1}}-T_{s-1,t-1,u_{s-1},v_{t}}\right)\\
			&\quad+\sum_{s=1}^{m}\sum_{t=1}^n\left(T_{s-1,t,u_{s-1},v}-T_{s-1,t-1,u_{s-1},v}- T_{s-1,t,u_{s-1},v_{t}}+T_{s-1,t-1,u_{s-1},v_{t}}\right)\\
			&\quad+\sum_{s=1}^{m}\sum_{t=1}^n\left(T_{s,t-1,u,v_{t-1}}^{*}-T_{s,t-1,u,v_{t}}^{*}\right)\\
			&\quad+\sum_{s=1}^{m}\sum_{t=1}^n\left(T_{s,t,u,v}^{*}-T_{s,t-1,u,v}^{*}- T_{s,t,u,v_{t}}^{*}+T_{s,t-1,u,v_{t}}^{*}\right)\\
			&:=I_1(m,n,u,v)+I_2(m,n,u,v)+I_3(m,n,u,v)+I_4(m,n,u,v).
		\end{split}
	\end{equation}
	By definitions of $u_s$ and $v_t$, we have either $u_{s-1}=u_s$ or $u_{s-1}=u_s+2^{s-1}$ and $v_{t-1}=v_t$ or $v_{t-1}=v_t+2^{t-1}$. 
	It is also easy to see that $0\le u_s\le u<u_s+2^s, 0\le v_t\le v<v_t+2^t$. Hereafter, the sum $\sum_{i=A+1}^{A}(\cdot)_i$ is interpreted to be $0$.
	Keeping these facts and conventions in mind, we have for all $1\le u<2^m,1\le v<2^n$, $m\ge 1,n\ge 1$, 
	\begin{equation}\label{mz33}
		\begin{split}
			\max_{\substack{1\le u< 2^m\\ 1\le v< 2^n}}|I_1(m,n,u,v)|&=\max_{\substack{1\le u< 2^m\\ 1\le v< 2^n}}\left|\sum_{s=1}^{m}\sum_{t=1}^n\left(\sum_{i=u_s+1}^{u_{s-1}} \sum_{j=v_t+1}^{v_{t-1}}\left(X_{s+t-2,i,j}-\E X_{s+t-2,i,j}\right) \right)\right|\\
			&\le \sum_{s=1}^m\sum_{t=1}^n \max_{\substack{0\le k< 2^{m-s}\\ 0\le \ell< 2^{n-t}}}\left|\sum_{i=k2^s+1}^{k2^s+2^{s-1}}\sum_{j=\ell 2^t+1}^{\ell 2^t+2^{t-1}}(X_{s+t-2,i,j}-\E X_{s+t-2,i,j})\right|.
		\end{split}
	\end{equation}
	Similarly, for all $1\le u<2^m,1\le v<2^n,m\ge1,n\ge1$, we have
	\begin{equation}\label{mz35}
		\begin{split}
			|I_2(m,n,u,v)|&=\left|\sum_{s=1}^{m}\sum_{t=1}^n\sum_{i=u_s+1}^{u_{s-1}} \sum_{j=v_t+1}^{v}\left(X_{s+t-1,i,j}-X_{s+t-2,i,j}-\E(X_{s+t-1,i,j}-X_{s+t-2,i,j})\right)\right|\\
			&\le \sum_{s=1}^{m}\sum_{t=1}^n\sum_{i=u_s+1}^{u_{s}+2^{s-1}} \sum_{j=v_t+1}^{v_t+2^t}\left(X_{s+t-1,i,j}^{*}+\E X_{s+t-1,i,j}^{*}\right)\\
			&=\sum_{s=1}^{m}\sum_{t=1}^n \sum_{i=u_s+1}^{u_{s}+2^{s-1}} \sum_{j=v_t+1}^{v_t+2^t} \left(Y_{s+t-1,i,j}^{*}+2\E X_{s+t-1,i,j}^{*}\right)\\
			&\le \sum_{s=1}^{m}\sum_{t=1}^n \left|\sum_{i=u_s+1}^{u_{s}+2^{s-1}} \sum_{j=v_t+1}^{v_t+2^t} Y_{s+t-1,i,j}^{*}\right|+2\sum_{s=1}^{m}\sum_{t=1}^{n}\sum_{i=u_s+1}^{u_{s}+2^{s-1}} \sum_{j=v_t+1}^{v_t+2^t} b_{2^{s+t}}\mathbb{P}\left(X_{i,j}>b_{2^{s+t-2}}\right),
		\end{split}
	\end{equation}
	where we have applied \eqref{max14} in the first and the last inequalities.
	Now, by recalling definitions of $u_s$ and $v_t$, we have from \eqref{mz35} that
	\begin{equation}\label{mz36}
		\begin{split}
			\max_{\substack{1\le u< 2^m\\ 1\le v< 2^n}}|I_2(m,n,u,v)|
			&\le \sum_{s=1}^m\sum_{t=1}^n \max_{\substack{0\le k< 2^{m-s}\\ 0\le \ell< 2^{n-t}}}\left|\sum_{i=k2^s+1}^{k2^s+2^{s-1}}\sum_{j=\ell 2^t+1}^{\ell 2^t+2^{t}}Y_{s+t-1,i,j}^{*}\right|\\
			&\qquad+\sum_{s=1}^{m}\sum_{t=1}^n 2^{s+t}b_{2^{s+t}}\max_{\substack{1\le i< 2^{m}\\ 1\le j< 2^{n}}}\mathbb{P}\left(X_{i,j}>b_{2^{s+t-2}}\right).
		\end{split}
	\end{equation}
	Similarly, we have
	\begin{equation}\label{mz37}
		\begin{split}
			\max_{\substack{1\le u< 2^m\\ 1\le v< 2^n}}|I_3(m,n,u,v)|
			&\le \sum_{s=1}^m\sum_{t=1}^n \max_{\substack{0\le k< 2^{m-s}\\ 0\le \ell< 2^{n-t}}}\left|\sum_{i=k2^s+1}^{k2^s+2^{s}}\sum_{j=\ell 2^t+1}^{\ell 2^t+2^{t-1}}Y_{s+t-1,i,j}^{*}\right|\\
			&\qquad+\sum_{s=1}^{m}\sum_{t=1}^n 2^{s+t}b_{2^{s+t}}\max_{\substack{1\le i< 2^{m}\\ 1\le j< 2^{n}}}\mathbb{P}\left(X_{i,j}>b_{2^{s+t-2}}\right),
		\end{split}
	\end{equation}
	and
	\begin{equation}\label{mz39}
		\begin{split}
			\max_{\substack{1\le u< 2^m\\ 1\le v< 2^n}}|I_4(m,n,u,v)|
			&\le \sum_{s=1}^m\sum_{t=1}^n \max_{\substack{0\le k< 2^{m-s}\\ 0\le \ell< 2^{n-t}}}\left|\sum_{i=k2^s+1}^{k2^s+2^{s}}\sum_{j=\ell 2^t+1}^{\ell 2^t+2^{t}}\left(Y_{s+t,i,j}^{*}+Y_{s+t-1,i,j}^{*}\right)\right|\\
			&\qquad+4\sum_{s=1}^{m}\sum_{t=1}^n 2^{s+t}b_{2^{s+t}}\max_{\substack{1\le i< 2^{m}\\ 1\le j< 2^{n}}}\mathbb{P}\left(X_{i,j}>b_{2^{s+t-2}}\right).
		\end{split}
	\end{equation}
	Combining \eqref{mz31d}, \eqref{mz33}, \eqref{mz36}--\eqref{mz39} yields \eqref{mz30}.
	
\end{proof}

%\bibliographystyle{elsarticle-harv} 
%\bibliography{mybib}
	
\end{document}